\documentclass[reqno]{amsart}
\usepackage{amsfonts}
\usepackage{amsmath}
\usepackage[english]{babel}
\usepackage{amsthm}
\usepackage{inputenc}
\usepackage{amssymb}
\usepackage{graphicx}
\usepackage{lineno}
\usepackage{float}

\newtheorem{thm}{Theorem}[section]
\newtheorem{cor}[thm]{Corollary}
\newtheorem{lem}[thm]{Lemma}

\newtheorem{defn}[thm]{Definition}
\newtheorem{exam}[thm]{Example}

\numberwithin{equation}{section}

\input{tcilatex}

\begin{document}
\title[Nilpotent evolution algebras]{\textbf{Five-dimensional nilpotent
evolution algebras}}
\author{A. S. Hegazi and Hani Abdelwahab}
\address[A. S. Hegazi -- Hani Abdelwahab]{Department of Mathematics, Faculty
of science, Mansoura university. (Egypt)}
\email{hegazi@mans.edu.eg --- haniamar1985@gmail.com}

\begin{abstract}
The paper is devoted to give a complete classification of five-dimension
nilpotent evolution algebras over an algebraically closed field. We obtained
a list of $27$ isolated non-isomorphic nilpotent evolution algebras and $2$
families of non-isomorphic algebras depending on one parameter.
\end{abstract}

\maketitle

\section{Introduction}

Evolution algebras were introduced in 2006 by Tian and Vojtechovsky \cite{V}
motivated by the evolution laws of genetics. In this sense, the
multiplication in an evolution algebra expresses self-reproduction of
non-Mendelian genetics. Evolution algebras are not defined by identities,
and hence they do not form a variety of non-associative algebras; they are
commutative, but not in general power-associative, and possess some
distinguishing properties that lead to many interesting mathematical results.

\bigskip

The classification, up to isomorphism, of any class of algebras is a
fundamental and very difficult problem. In \cite{Hegazi}\ the authors
classified nilpotent evolution algebras of dimension up to four over an
algebraically closed field. The aim of this paper is to give a full
classification of nilpotent evolution algebras of dimension up to five over
an algebraically closed field $\mathbb{F}$.

\bigskip

The paper is organized as follows. In Section \ref{the metod} we describe
the method that we use to classify nilpotent evolution algebras, which also
appeared in \cite{Hegazi}. \ In Section\ \ref{four}\ we list, up to
isomorphism, all nilpotent evolution algebras of dimension up to four. In
Section \ref{classification}\ the classification of five-dimensional
nilpotent evolution algebras is given.

\section{The method}

\label{the metod}

\begin{defn}
\emph{(See \cite{T}.)} An \emph{evolution algebra} is an algebra ${\mathcal{E%
}}$ containing a basis (as a vector space) $B=\left\{ e_{1},\ldots
,e_{n}\right\} $ such that $e_{i}e_{j}=0$ for any $1\leq i<j\leq n$. A basis
with this property is called a \emph{natural basis}.
\end{defn}

\begin{defn}
An \emph{ideal} $\mathcal{I}$ of an evolution algebra ${\mathcal{E}}$ is an
evolution algebra satisfying ${\mathcal{E}}\mathcal{I}\subseteq \mathcal{I}$.
\end{defn}

\smallskip Given an evolution algebra ${\mathcal{E}}$, consider its \emph{%
annihilator }%
\begin{equation*}
ann\left( {\mathcal{E}}\right) :=\left\{ x\in {\mathcal{E}}:x{\mathcal{E}}%
=0\right\} .
\end{equation*}

\begin{lem}
\emph{(See \cite{Elduque}.)} Let $B=\left\{ e_{1},\ldots ,e_{n}\right\} $ be
a natural basis of an evolution algebra ${\mathcal{E}}$. Then 
\begin{equation*}
ann\left( {\mathcal{E}}\right) =\mathrm{span}\left\{
e_{i}:e_{i}^{2}=0\right\} .
\end{equation*}
\end{lem}

Given an evolution algebra ${\mathcal{E}}$, we introduce the following
sequence of subspaces:%
\begin{equation*}
\begin{array}{cc}
{\mathcal{E}}^{\left\langle 1\right\rangle }={\mathcal{E}}, & {\mathcal{E}}%
^{\left\langle k+1\right\rangle }={\mathcal{E}}^{\left\langle k\right\rangle
}{\mathcal{E}}.%
\end{array}%
\end{equation*}

\begin{defn}
An algebra ${\mathcal{E}}$ is called \emph{nilpotent} if there exists $n\in {%
\mathbb{N}}$ such that ${\mathcal{E}}^{\left\langle n\right\rangle }=0$, and
the minimal such number is called the \emph{index of nilpotency}.
\end{defn}

Let ${\mathcal{E}}$ be an evolution algebra with a natural basis $B_{{%
\mathcal{E}}}=\left\{ e_{1},\ldots ,e_{m}\right\} $ and $V$ be a vector
space with a basis $B_{V}=\left\{ e_{m+1},\ldots ,e_{n}\right\} $. Let ${%
\mathcal{Z}}\left( {\mathcal{E\times E}},V\right) $ be the set of all
bilinear maps $\theta :{\mathcal{E\times E\longrightarrow }}V$ such that $%
\theta \left( e_{i},e_{j}\right) =0$ for all $i\neq j\mathcal{\ }$. Set ${%
\mathcal{E}}_{\theta }={\mathcal{E\oplus }}V$, and then define a
multiplication on ${\mathcal{E}}_{\theta }$ by $e_{i}\star
e_{j}=e_{i}e_{j}\mid _{{\mathcal{E}}}+\theta \left( e_{i},e_{j}\right) $ if $%
1\leq i,j\leq m$, and otherwise $e_{i}\star e_{j}=0$. Then ${\mathcal{E}}%
_{\theta }$ is an evolution algebra with a natural basis $B_{{\mathcal{E}}%
_{\theta }}=\left\{ e_{1},\ldots ,e_{n}\right\} $. The evolution algebra ${%
\mathcal{E}}_{\theta }$\ is called a $\left( n-m\right) $\emph{-dimensional
annihilator extension} of ${\mathcal{E}}$ by $V$.

\bigskip

Let $\theta :{\mathcal{E\times E}}\longrightarrow V$ be a symmetric bilinear
map. Then the set $\theta ^{\bot }=\left\{ x\in {\mathcal{E}}:\theta (x,{%
\mathcal{E}})=0\right\} $ is called the \emph{radical} of $\theta $.

\begin{lem}
\emph{(See \cite[Lemma 11]{Hegazi}.)}\label{Ann(J+V)} Let $\theta \in {%
\mathcal{Z}}\left( {\mathcal{E\times E}},V\right) $. Then $ann({\mathcal{E}}%
_{\theta })=\left( \theta ^{\bot }\cap ann\left( {\mathcal{E}}\right)
\right) \oplus V.$
\end{lem}

So, if $V$ is a non-trivial vector space then the annihilator of ${\mathcal{E%
}}_{\theta }$\ will be non-trivial.

\begin{thm}
Let ${\mathcal{E}}$ be an evolution algebra with $ann\left( {\mathcal{E}}%
\right) \neq 0$. Then there exist, up to isomorphism, a unique evolution
algebra $E$, and a $\theta \in {\mathcal{Z}}\left( E\times E,V\right) $ with 
$\theta ^{\bot }\cap ann\left( E\right) =0$ such that ${\mathcal{E}}\cong
E_{\theta }$ and ${\mathcal{E}}/ann\left( {\mathcal{E}}\right) \cong E$.
\end{thm}

\begin{proof}
Let $\left\{ e_{1},\ldots ,e_{n}\right\} $ be a natural basis of ${\mathcal{E%
}}$ such that $\left\{ e_{m+1},\ldots ,e_{n}\right\} $ be a basis of $%
ann\left( {\mathcal{E}}\right) $. Set $V=ann\left( {\mathcal{E}}\right) $,
and let $E$ be a complement of $V$ in ${\mathcal{E}}$ (i.e., ${\mathcal{E}}%
=E\oplus V$). Let $P:{\mathcal{E}}\longrightarrow E$ be the projection of ${%
\mathcal{E}}$ onto $E$ (i.e., $P\left( e_{i}\right) =e_{i}$ if $1\leq i\leq
m $, and otherwise $P\left( e_{i}\right) =0$). Define a multiplication on $E$
by $e_{i}\ast _{E}e_{j}=P\left( e_{i}e_{j}\right) $ for $1\leq i,j\leq m$.
Then $E$ is an evolution algebra. Moreover, $P\left( e_{i}e_{j}\right)
=P\left( e_{i}\right) \ast _{E}P\left( e_{j}\right) $ for $1\leq i,j\leq n$.
Hence $P$ is a homomorphism of evolution algebras. So, ${\mathcal{E}}/V\cong
E$. Define a bilinear map $\theta :E\times E\longrightarrow V$ by $\theta
\left( e_{i},e_{j}\right) =$ $e_{i}e_{j}-P\left( e_{i}e_{j}\right) $ for $%
1\leq i,j\leq m$. Then $\theta \in {\mathcal{Z}}\left( E\times E,V\right) $
and $\theta ^{\bot }\cap ann\left( E\right) =0$. Therefore, $E_{\theta }$ is
an evolution algebra. Further, $e_{i}\star _{E_{\theta }}e_{j}=e_{i}\ast
_{E}e_{j}+\theta \left( e_{i},e_{j}\right) =e_{i}e_{j}$ for all $1\leq
i,j\leq m$. So, $E_{\theta }$ is the same evolution algebra as ${\mathcal{E}}
$.
\end{proof}

Now, for a linear map $f\in Hom\left( {\mathcal{E}},V\right) $, if we define 
$\delta f:{\mathcal{E\times E\longrightarrow }}V$ by $\delta f\left(
e_{i},e_{j}\right) =f\left( e_{i}e_{j}\right) $, then $\delta f\in {\mathcal{%
Z}}\left( {\mathcal{E\times E}},V\right) $. Let ${\mathcal{B}}\left( {%
\mathcal{E\times E}},V\right) =\left\{ \delta f:f\in Hom\left( {\mathcal{E}}%
,V\right) \right\} $. Then ${\mathcal{B}}\left( {\mathcal{E\times E}}%
,V\right) $\ is a subspace of ${\mathcal{Z}}\left( {\mathcal{E\times E}}%
,V\right) $. We define ${\mathcal{H}}\left( {\mathcal{E\times E}},V\right) $
to be the quotient space ${\mathcal{Z}}\left( {\mathcal{E\times E}},V\right) %
\big/{\mathcal{B}}\left( {\mathcal{E\times E}},V\right) $. The equivalence
class of $\theta \in {\mathcal{Z}}\left( {\mathcal{E\times E}},V\right) $ is
denoted $\left[ \theta \right] \in {\mathcal{H}}\left( {\mathcal{E\times E}}%
,V\right) $.

\begin{lem}
\label{equivalence class}Let $\theta ,\vartheta \in {\mathcal{Z}}\left( {%
\mathcal{E\times E}},V\right) $ such that $\left[ \theta \right] =\left[
\vartheta \right] $. Then

\begin{enumerate}
\item $\theta ^{\bot }\cap ann\left( {\mathcal{E}}\right) =\vartheta ^{\bot
}\cap ann\left( {\mathcal{E}}\right) $ or, equivalently, $ann\left( {%
\mathcal{E}}_{\theta }\right) =ann\left( {\mathcal{E}}_{\vartheta }\right) $.

\item ${\mathcal{E}}_{\theta }\cong {\mathcal{E}}_{\vartheta }$.
\end{enumerate}
\end{lem}

\begin{proof}
\begin{enumerate}
\item Since $\left[ \theta \right] =\left[ \vartheta \right] $, $\vartheta
=\theta +\delta f$ for some $f\in Hom\left( {\mathcal{E}},V\right) $. So $%
\vartheta \left( x,y\right) =\theta \left( x,y\right) +f\left( xy\right) $
for all $x,y\in {\mathcal{E}}$. Hence $\theta \left( x,y\right) =xy=0$ if
and only if $\vartheta \left( x,y\right) =xy=0$. Therefore $\theta ^{\bot
}\cap ann\left( {\mathcal{E}}\right) =\vartheta ^{\bot }\cap ann\left( {%
\mathcal{E}}\right) $. Then, by Lemma \ref{Ann(J+V)}, $ann\left( {\mathcal{E}%
}_{\theta }\right) =ann\left( {\mathcal{E}}_{\vartheta }\right) $.

\item See \cite[Lemma 5]{Hegazi}.
\end{enumerate}
\end{proof}

Let ${\mathcal{E}}$ be an evolution algebra with a natural basis $B_{{%
\mathcal{E}}}=\left\{ e_{1},\ldots ,e_{m}\right\} $ and $V$ be a vector
space with a basis $B_{V}=\left\{ e_{m+1},\ldots ,e_{n}\right\} $. Then a $%
\theta \in {\mathcal{Z}}\left( {\mathcal{E\times E}},V\right) $ can be
uniquely written as $\theta \left( e_{i},e_{i}\right) =\underset{j=1}{%
\overset{n-m}{\sum }}\theta _{j}\left( e_{i},e_{i}\right) e_{m+j}$, where $%
\theta _{j}\in {\mathcal{Z}}\left( {\mathcal{E\times E}},\mathbb{F}\right) $%
. Moreover, $\theta \in {\mathcal{B}}\left( {\mathcal{E\times E}},V\right) $
if and only if all $\theta _{j}\in {\mathcal{B}}\left( {\mathcal{E\times E}},%
\mathbb{F}\right) $. Furthermore, $\theta ^{\bot }=\theta _{1}^{\bot }\cap
\ldots \cap $ $\theta _{n-m}^{\bot }$.\bigskip

\begin{lem}
\emph{(See \cite[Lemma 6]{Hegazi}.) }\label{Qouitent}Let ${\mathcal{E}}$ be
an n-dimensional nilpotent evolution algebra. Let $e_{1},\ldots ,e_{m}$ be a
basis of ${\mathcal{E}}^{\left\langle 2\right\rangle }$. Then $\dim {%
\mathcal{H}}\left( {\mathcal{E\times E}},\mathbb{F}\right) =n-m$ and ${%
\mathcal{B}}\left( {\mathcal{E\times E}},\mathbb{F}\right) =\left\langle
\delta e_{1}^{\ast },\ldots ,\delta e_{m}^{\ast }\right\rangle $ where $%
e_{i}^{\ast }(e_{j})=\delta _{ij}$ and $\delta _{ij}$ is the \emph{Kronecker
delta}.
\end{lem}

Let ${\mathcal{E}}$ be an evolution algebra with a natural basis $B_{{%
\mathcal{E}}}=\left\{ e_{1},\ldots ,e_{m}\right\} $. We define $\delta
_{e_{i},e_{j}}$ to be the symmetric bilinear form $\delta _{e_{i},e_{j}}:{%
\mathcal{E\times E\longrightarrow }}\mathbb{F}$ with $\delta
_{e_{i},e_{j}}\left( e_{l},e_{m}\right) =1$ if $\left\{ i,j\right\} =\left\{
l,m\right\} $, and otherwise $\delta _{e_{i},e_{j}}\left( e_{l},e_{m}\right)
=0$. Denote the space of all symmetric bilinear forms on ${\mathcal{E}}$ by $%
Sym\left( {\mathcal{E}}\right) $. Then $Sym\left( {\mathcal{E}}\right)
=\left\langle \delta _{e_{i},e_{j}}:1\leq i\leq j\leq m\right\rangle $ while 
${\mathcal{Z}}\left( {\mathcal{E\times E}},\mathbb{F}\right) =\left\langle
\delta _{e_{i},e_{i}}:1=1,2,\ldots ,m\right\rangle $.

\begin{exam}
Let ${\mathcal{E}}:e_{1}^{2}=e_{2}$ be a 2-dimensional nilpotent evolution
algebra. Then ${\mathcal{Z}}\left( {\mathcal{E\times E}},\mathbb{F}\right)
=\left\langle \delta _{e_{1},e_{1}},\delta _{e_{2},e_{2}}\right\rangle $. By
Lemma \ref{Qouitent}, ${\mathcal{B}}\left( {\mathcal{E\times E}},\mathbb{F}%
\right) =\left\langle \delta e_{1}^{\ast }\right\rangle $. Since ${\mathcal{B%
}}\left( {\mathcal{E\times E}},\mathbb{F}\right) $ is a subspace of ${%
\mathcal{Z}}\left( {\mathcal{E\times E}},\mathbb{F}\right) $, $\delta
e_{2}^{\ast }=\alpha \delta _{e_{1},e_{1}}+\beta \delta _{e_{2},e_{2}}$.
Then $\alpha =\delta e_{2}^{\ast }\left( e_{1},e_{1}\right) =e_{1}^{\ast
}\left( e_{1}^{2}\right) =1$ and $\beta =\delta e_{2}^{\ast }\left(
e_{2},e_{2}\right) =e_{1}^{\ast }\left( e_{2}^{2}\right) =0$. So, ${\mathcal{%
B}}\left( {\mathcal{E\times E}},\mathbb{F}\right) =\left\langle \delta
_{e_{1},e_{1}}\right\rangle $ and therefore ${\mathcal{H}}\left( {\mathcal{%
E\times E}},\mathbb{F}\right) =\left\langle \left[ \delta _{e_{2},e_{2}}%
\right] \right\rangle $.
\end{exam}

Let $Aut\left( {\mathcal{E}}\right) $ be the automorphism group of the
algebra ${\mathcal{E}}$. Let $\phi \in Aut\left( {\mathcal{E}}\right) $. For 
$\theta \in {\mathcal{Z}}\left( {\mathcal{E\times E}},V\right) $ defining $%
\phi \theta \left( e_{i},e_{j}\right) =\theta \left( \phi \left(
e_{i}\right) ,\phi \left( e_{j}\right) \right) $ then $\phi \theta \in {%
\mathcal{Z}}\left( {\mathcal{E\times E}},V\right) $ if and only if $\theta
\left( \phi \left( e_{i}\right) ,\phi \left( e_{j}\right) \right) =0$ for
any $i\neq j$. Consider the matrix representation and assume that $\theta =%
\overset{m}{\underset{i=1}{\sum }}c_{ii}\delta _{e_{i},e_{i}}\in {\mathcal{Z}%
}\left( {\mathcal{E\times E}},\mathbb{F}\right) $. Then $\theta $ can be
represented by an $m\times m$ diagonal matrix $\big(c_{ii}\big)$, and then $%
\phi \theta =\phi ^{t}\big(c_{ii}\big)\phi $. So, $\phi \theta \in {\mathcal{%
Z}}\left( {\mathcal{E\times E}},V\right) $ if and only if the matrix
represents $\phi \theta $\ is diagonal. Define a subgroup $\mathcal{S}%
_{\theta }\left( {\mathcal{E}}\right) $ of $Aut\left( {\mathcal{E}}\right) $
by 
\begin{equation*}
\mathcal{S}_{\theta }\left( {\mathcal{E}}\right) =\left\{ \phi \in Aut\left( 
{\mathcal{E}}\right) :\theta \left( \phi \left( e_{i}\right) ,\phi \left(
e_{j}\right) \right) =0\text{ \ }\forall i\neq j\right\} .
\end{equation*}%
Then, for each $\phi \in \mathcal{S}_{\theta }\left( {\mathcal{E}}\right) $, 
$\phi \theta \in {\mathcal{Z}}\left( {\mathcal{E\times E}},V\right) $ and
therefore ${\mathcal{E}}_{\phi \theta }$ is an evolution algebra. Further, $%
\phi \theta \in {\mathcal{B}}\left( {\mathcal{E\times E}},V\right) $ if and
only if $\theta \in {\mathcal{B}}\left( {\mathcal{E\times E}},V\right) $ and
hence $\left[ \phi \theta \right] $ is a well-defined element of ${\mathcal{H%
}}\left( {\mathcal{E\times E}},V\right) $.

\begin{lem}
\label{iso1}Let $\theta \in {\mathcal{Z}}\left( {\mathcal{E\times E}}%
,V\right) $ and $\phi \in \mathcal{S}_{\theta }\left( {\mathcal{E}}\right) $%
. Then ${\mathcal{E}}_{\phi \theta }\cong {\mathcal{E}}_{\theta }$.
\end{lem}

\begin{proof}
Define a linear map $\sigma :{\mathcal{E}}_{\phi \theta }\longrightarrow {%
\mathcal{E}}_{\theta }$ by $\sigma (x+v)=\phi \left( x\right) +v$ for $x\in {%
\mathcal{E}}$ and $v\in V$. Then $\sigma $ is a bijective map. Moreover,%
\begin{eqnarray*}
\sigma \left( (x+v)(y+w)\right) &=&\sigma \left( xy+\phi \theta \left(
x,y\right) \right) \\
&=&\phi \left( xy\right) +\phi \theta \left( x,y\right) \\
&=&\phi \left( x\right) \phi \left( y\right) +\theta (\phi \left( x\right)
,\phi \left( y\right) )) \\
&=&\sigma \left( x+v\right) \sigma \left( y+w\right)
\end{eqnarray*}%
for any $x,y\in {\mathcal{E}}$ and $v,w\in V$. Hence $\sigma $ is an
isomorphism.
\end{proof}

Let $GL\left( V\right) $ be the set of bijective linear maps $%
V\longrightarrow V$. Let $\psi \in GL\left( V\right) $. For $\theta \in {%
\mathcal{Z}}\left( {\mathcal{E\times E}},V\right) $ define $\psi \theta
\left( e_{i},e_{j}\right) =\psi \left( \theta \left( e_{i},e_{j}\right)
\right) $. Then $\psi \theta \in {\mathcal{Z}}\left( {\mathcal{E\times E}}%
,V\right) $. So, $GL\left( V\right) $ acts on ${\mathcal{Z}}\left( {\mathcal{%
E\times E}},V\right) $. Also, $\theta \in {\mathcal{B}}\left( {\mathcal{%
E\times E}},V\right) $ if and only if $\psi \theta \in {\mathcal{B}}\left( {%
\mathcal{E\times E}},V\right) $. Hence $GL\left( V\right) $ acts on ${%
\mathcal{H}}\left( {\mathcal{E\times E}},V\right) $.

\begin{lem}
\label{iso2}Let $\theta \in {\mathcal{Z}}\left( {\mathcal{E\times E}}%
,V\right) $ and $\psi \in GL(V)$. Then ${\mathcal{E}}_{\theta }\cong {%
\mathcal{E}}_{\psi \theta }$.
\end{lem}

\begin{proof}
Define a linear map $\sigma :{\mathcal{E}}_{\theta }\longrightarrow {%
\mathcal{E}}_{\psi \theta }$ by $\sigma (x+v)=x+\psi \left( v\right) $ for $%
x\in {\mathcal{E}}$ and $v\in V$. Then $\sigma $ is a bijective map.
Moreover,%
\begin{equation*}
\sigma \left( (x+v)(y+w)\right) =\sigma \left( xy+\theta \left( x,y\right)
\right) =xy+\psi \left( \theta \left( x,y\right) \right) =\sigma \left(
x+v\right) \sigma \left( y+w\right) 
\end{equation*}%
for any $x,y\in {\mathcal{E}}$ and $v,w\in V$. Hence $\sigma $ is an
isomorphism.
\end{proof}

\begin{lem}
\label{iso3}Let $\theta ,\vartheta \in {\mathcal{Z}}\left( {\mathcal{E\times
E}},V\right) $. If there exist a map $\phi \in Aut\left( {\mathcal{E}}%
\right) $ and a map $\psi \in GL(V)$ such that $\left[ \phi \theta \right] =%
\left[ \psi \vartheta \right] $, then ${\mathcal{E}}_{\theta }\cong {%
\mathcal{E}}_{\vartheta }$.
\end{lem}

\begin{proof}
Let $\phi \in Aut\left( {\mathcal{E}}\right) $ and $\psi \in GL(V)$ such
that $\left[ \phi \theta \right] =\left[ \psi \vartheta \right] $ (i.e., $%
\phi \in \mathcal{S}_{\theta }\left( {\mathcal{E}}\right) $). Then, by Lemma %
\ref{equivalence class}, ${\mathcal{E}}_{\phi \theta }\cong {\mathcal{E}}%
_{\psi \vartheta }$. Moreover, by Lemma \ref{iso1} and Lemma \ref{iso2}, ${%
\mathcal{E}}_{\theta }\cong {\mathcal{E}}_{\phi \theta }\cong {\mathcal{E}}%
_{\psi \vartheta }\cong {\mathcal{E}}_{\vartheta }$.
\end{proof}

\begin{lem}
\emph{(See \cite[Lemma 14]{Hegazi}.)}\label{same orbit} Let ${\mathcal{E}}$
be an evolution algebra with a natural basis $B_{{\mathcal{E}}}=\left\{
e_{1},\ldots ,e_{m}\right\} $, $V$ be a vector space with a basis $%
B_{V}=\left\{ e_{m+1},\ldots ,e_{n}\right\} $, and $\theta ,\vartheta \in {%
\mathcal{Z}}\left( {\mathcal{E\times E}},V\right) $ with $\theta ^{\perp
}\cap ann\left( {\mathcal{E}}\right) =\vartheta ^{\perp }\cap ann\left( {%
\mathcal{E}}\right) =0$. Suppose that $\theta \left( e_{i},e_{i}\right) =%
\underset{j=1}{\overset{n-m}{\sum }}\theta _{j}\left( e_{i},e_{i}\right)
e_{m+j}$ and $\vartheta \left( e_{i},e_{i}\right) =\underset{j=1}{\overset{%
n-m}{\sum }}\vartheta _{j}\left( e_{i},e_{i}\right) e_{m+j}$. Then ${%
\mathcal{E}}_{\theta }\cong {\mathcal{E}}_{\vartheta }$ if and only if there
is a $\phi \in Aut\left( {\mathcal{E}}\right) $ (more precisely, $\phi \in 
\mathcal{S}_{\vartheta }\left( {\mathcal{E}}\right) $) such that the $\left[
\phi \vartheta _{i}\right] $ span the same subspace of ${\mathcal{H}}\left( {%
\mathcal{E\times E}},\mathbb{F}\right) $ as the $\left[ \theta _{i}\right] $.
\end{lem}

\begin{defn}
\label{annihilator component}Let ${\mathcal{E}}$ be an evolution algebra
with a natural basis $B_{{\mathcal{E}}}=\left\{ e_{1},\ldots ,e_{m}\right\} $%
. If ${\mathcal{E}}=\left\langle
e_{1},..,e_{r-1},e_{r+1},..,e_{m}\right\rangle \oplus \mathbb{F}e_{r}$ is
the direct sum of two nilpotent ideals, then $\mathbb{F}e_{r}$ is called an 
\emph{annihilator component} of ${\mathcal{E}}$.
\end{defn}

\begin{lem}
\emph{(See \cite[Lemma 16]{Hegazi}.)} Let ${\mathcal{E}}$ be an evolution
algebra with a natural basis $B_{{\mathcal{E}}}=\left\{ e_{1},\ldots
,e_{m}\right\} $, $V$ be a vector space with a basis $B_{V}=\left\{
e_{m+1},\ldots ,e_{n}\right\} $, and $\theta \in {\mathcal{Z}}\left( {%
\mathcal{E\times E}},V\right) $. Suppose that $\theta \left(
e_{i},e_{i}\right) =\underset{j=1}{\overset{n-m}{\sum }}\theta _{j}\left(
e_{i},e_{i}\right) e_{m+j}$ and $\theta ^{\bot }\cap ann\left( {\mathcal{E}}%
\right) =0$. Then ${\mathcal{E}}_{\theta }$ has an annihilator component if
and only if$\mathcal{\ }\left[ \theta _{1}\right] ,\left[ \theta _{2}\right]
,\ldots ,\left[ \theta _{s}\right] $ are linearly dependent in ${\mathcal{H}}%
\left( {\mathcal{E\times E}},\mathbb{F}\right) $.
\end{lem}

\begin{lem}
Let $\phi \in Aut\left( {\mathcal{E}}\right) $ and $\theta \in {\mathcal{Z}}%
\left( {\mathcal{E\times E}},\mathbb{F}\right) $. Then $\dim \left( \phi
\theta \right) ^{\bot }=\dim \theta ^{\bot }$.
\end{lem}

\begin{proof}
Since $\phi \theta \left( x,{\mathcal{E}}\right) =\theta \left( \phi \left(
x\right) ,{\mathcal{E}}\right) $, $x\in \left( \phi \theta \right) ^{\bot }$
if and only if $\phi \left( x\right) \in \theta ^{\bot }$. Define a linear
map $\sigma :\left( \phi \theta \right) ^{\bot }\longrightarrow \theta
^{\bot }$ by $\sigma \left( x\right) =\phi \left( x\right) $. Then $\sigma $%
\ is a bijective map. So, $\dim \left( \phi \theta \right) ^{\bot }=\dim
\theta ^{\bot }$.
\end{proof}

For $\theta _{1},\theta _{2},\ldots ,\theta _{n-m}\in {\mathcal{Z}}\left( {%
\mathcal{E\times E}},\mathbb{F}\right) $, let $\Psi \left( \theta
_{1},\theta _{2},\ldots ,\theta _{n-m}\right) =\left( m_{1},m_{2},\ldots
,m_{n-m}\right) $ be the ordered descending sequence of $\dim \theta
_{1}^{\perp },\dim \theta _{2}^{\perp },...,\dim \theta _{n-m}^{\perp }$.

\begin{cor}
\label{inv}Let $\phi \in Aut\left( {\mathcal{E}}\right) $ and $\theta
_{1},\ldots ,\theta _{n-m}\in {\mathcal{Z}}\left( {\mathcal{E\times E}},%
\mathbb{F}\right) $. Then $\Psi \left( \theta _{1},\ldots ,\theta
_{n-m}\right) =\Psi \left( \phi \theta _{1},\ldots ,\phi \theta
_{n-m}\right) $.
\end{cor}

Let $\left[ \theta _{1}\right] ,\ldots ,\left[ \theta _{n-m}\right] \in {%
\mathcal{H}}\left( {\mathcal{E\times E}},\mathbb{F}\right) $. We define $%
\Psi \left( \left[ \theta _{1}\right] ,\ldots ,\left[ \theta _{n-m}\right]
\right) $ to be the least upper bound for the lexicographic order of the set%
\begin{equation*}
\left\{ \Psi \left( \vartheta _{1},\ldots ,\vartheta _{n-m}\right)
:\vartheta _{1},\ldots ,\vartheta _{n-m}\in {\mathcal{Z}}\left( {\mathcal{%
E\times E}},\mathbb{F}\right) \mbox{ and }\left\langle \left[ \theta _{1}%
\right] ,...,\left[ \theta _{n-m}\right] \right\rangle =\left\langle \left[
\vartheta _{1}\right] ,...,\left[ \vartheta _{n-m}\right] \right\rangle
\right\} .
\end{equation*}

\begin{lem}
Let ${\mathcal{E}}$ be an evolution algebra with a natural basis $B_{{%
\mathcal{E}}}=\left\{ e_{1},\ldots ,e_{m}\right\} $, $V$ be a vector space
with a basis $B_{V}=\left\{ e_{m+1},\ldots ,e_{n}\right\} $, and $\theta
,\vartheta \in {\mathcal{Z}}\left( {\mathcal{E\times E}},V\right) $ with $%
\theta ^{\perp }\cap ann\left( {\mathcal{E}}\right) =\vartheta ^{\perp }\cap
ann\left( {\mathcal{E}}\right) =0$. Suppose that $\theta (e_{l},e_{l})=%
\underset{i=1}{\overset{n-m}{\sum }}\theta _{i}(e_{l},e_{l})e_{m+i}$ and $%
\vartheta (e_{l},e_{l})=\underset{i=1}{\overset{n-m}{\sum }}\vartheta
_{i}(e_{l},e_{l})e_{m+i}$. If ${\mathcal{E}}_{\theta }\cong {\mathcal{E}}%
_{\vartheta }$ then $\Psi \left( \left\langle \left[ \theta _{1}\right]
,\ldots ,\left[ \theta _{n-m}\right] \right\rangle \right) =\Psi \left(
\left\langle \left[ \vartheta _{1}\right] ,\ldots ,\left[ \vartheta _{n-m}%
\right] \right\rangle \right) $.
\end{lem}

\begin{proof}
Let ${\mathcal{E}}_{\theta }\cong {\mathcal{E}}_{\vartheta }$. Then by Lemma %
\ref{same orbit}, there exists $\phi \in Aut\left( {\mathcal{E}}\right) $
such that $\left\langle \left[ \phi \vartheta _{1}\right] ,\ldots ,\left[
\phi \vartheta _{n-m}\right] \right\rangle =\left\langle \left[ \theta _{1}%
\right] ,\ldots ,\left[ \theta _{n-m}\right] \right\rangle $. Then $\Psi
\left( \left\langle \left[ \phi \vartheta _{1}\right] ,\ldots ,\left[ \phi
\vartheta _{n-m}\right] \right\rangle \right) =\Psi \left( \left\langle %
\left[ \theta _{1}\right] ,\ldots ,\left[ \theta _{n-m}\right] \right\rangle
\right) $. Further, let $\Psi \left( \left[ \phi \vartheta _{1}\right]
,\ldots ,\left[ \phi \vartheta _{n-m}\right] \right) =\Psi \left( \eta
_{1},\ldots ,\eta _{n-m}\right) $ for some $\eta _{1},\ldots ,\eta _{n-m}\in 
{\mathcal{Z}}\left( {\mathcal{E\times E}},\mathbb{F}\right) $. Then $%
\left\langle \left[ \phi \vartheta _{1}\right] ,\ldots ,\left[ \phi
\vartheta _{n-m}\right] \right\rangle =\left\langle \left[ \eta _{1}\right]
,\ldots ,\left[ \eta _{n-m}\right] \right\rangle $ and therefore $%
\left\langle \left[ \vartheta _{1}\right] ,\ldots ,\left[ \vartheta _{n-m}%
\right] \right\rangle =\left\langle \left[ \phi ^{-1}\eta _{1}\right]
,\ldots ,\left[ \phi ^{-1}\eta _{n-m}\right] \right\rangle $. By Corollary %
\ref{inv}, $\Psi \left( \eta _{1},\ldots ,\eta _{n-m}\right) =\Psi \left(
\phi ^{-1}\eta _{1},\ldots ,\phi ^{-1}\eta _{n-m}\right) $ and hence $\Psi
\left( \left\langle \left[ \vartheta _{1}\right] ,\ldots ,\left[ \vartheta
_{n-m}\right] \right\rangle \right) \geq \Psi \left( \eta _{1},\ldots ,\eta
_{n-m}\right) $. So, $\Psi \left( \left\langle \left[ \vartheta _{1}\right]
,\ldots ,\left[ \vartheta _{n-m}\right] \right\rangle \right) \geq \Psi
\left( \left[ \phi \vartheta _{1}\right] ,\ldots ,\left[ \phi \vartheta
_{n-m}\right] \right) $. On the other hand, let $\Psi \left( \left\langle %
\left[ \vartheta _{1}\right] ,\ldots ,\left[ \vartheta _{n-m}\right]
\right\rangle \right) =\Psi \left( \eta _{1}^{\prime },\ldots ,\eta
_{n-m}^{\prime }\right) $ for some $\eta _{1}^{\prime },\ldots ,\eta
_{n-m}^{\prime }\in {\mathcal{Z}}\left( {\mathcal{E\times E}},\mathbb{F}%
\right) $. Then $\left\langle \left[ \vartheta _{1}\right] ,\ldots ,\left[
\vartheta _{n-m}\right] \right\rangle =\left\langle \left[ \eta _{1}^{\prime
}\right] ,\ldots ,\left[ \eta _{n-m}^{\prime }\right] \right\rangle $ and
therefore $\left\langle \left[ \phi \vartheta _{1}\right] ,\ldots ,\left[
\phi \vartheta _{n-m}\right] \right\rangle =\left\langle \left[ \phi \eta
_{1}^{\prime }\right] ,\ldots ,\left[ \phi \eta _{n-m}^{\prime }\right]
\right\rangle $. So, by Corollary \ref{inv}, $\Psi \left( \left\langle \left[
\vartheta _{1}\right] ,\ldots ,\left[ \vartheta _{n-m}\right] \right\rangle
\right) \leq \Psi \left( \left\langle \left[ \phi \vartheta _{1}\right]
,\ldots ,\left[ \phi \vartheta _{n-m}\right] \right\rangle \right) $. This
completes the proof.
\end{proof}

Now we carry out a procedure that takes as input a nilpotent evolution
algebra ${\mathcal{E}}$\ of dimension $m$ with a natural basis $B_{{\mathcal{%
E}}}=\left\{ e_{1},\ldots ,e_{m}\right\} $. It gives us all nilpotent
evolution algebras ${\mathcal{\tilde{E}}}$ of dimension $n$ with a natural
basis $B_{{\mathcal{\tilde{E}}}}=\left\{ e_{1},\ldots ,e_{m},e_{m+1},\ldots
,e_{n}\right\} $ such that ${\mathcal{\tilde{E}}}/ann\left( {\mathcal{\tilde{%
E}}}\right) \cong {\mathcal{E}}$ and ${\mathcal{\tilde{E}}}$ has no
annihilator components. The procedure runs as follows:

\begin{enumerate}
\item Determine ${\mathcal{Z}}\left( {\mathcal{E\times E}},\mathbb{F}\right)
,{\mathcal{B}}\left( {\mathcal{E\times E}},\mathbb{F}\right) $ and ${%
\mathcal{H}}\left( {\mathcal{E\times E}},\mathbb{F}\right) $.

\item Let $V$ be a vector space with a basis $B_{V}=\left\{ e_{m+1},\ldots
,e_{n}\right\} $. Consider $\left[ \theta \right] \in {\mathcal{H}}\left( {%
\mathcal{E\times E}},V\right) $ with 
\begin{equation*}
\theta \left( e_{i},e_{i}\right) =\underset{j=1}{\overset{n-m}{\sum }}\theta
_{j}\left( e_{i},e_{i}\right) e_{m+j},
\end{equation*}%
where the $\left[ \theta _{j}\right] \in {\mathcal{H}}\left( {\mathcal{%
E\times E}},\mathbb{F}\right) $ are linearly independent, and $\theta ^{\bot
}\cap ann\left( {\mathcal{E}}\right) =0$. Use the action of $\mathcal{S}%
_{\theta }\left( {\mathcal{E}}\right) $ on $\left[ \theta \right] $ by hand
calculations to get a list of orbit representatives as small as possible.

\item For each $\theta $ found, construct ${\mathcal{E}}_{\theta }$. Get rid
of isomorphic ones.
\end{enumerate}

\section{Nilpotent evolution algebras of dimension up to four}

\label{four}Here we list all nilpotent evolution algebras over an
algebraically closed field $\mathbb{F}$ of dimension up to four. We denote
the $j$-th algebra of dimension $i$ with a natural basis $\left\{
e_{1},\ldots ,e_{i}\right\} $ by ${\mathcal{E}}_{i,j}$, and among the
natural basis of ${\mathcal{E}}_{i,j}$ the multiplication is specified by
giving only the nonzero products. Throughout the paper, we describe the
automorphism group, $Aut\left( {\mathcal{E}}\right) $, of an evolution
algebra ${\mathcal{E}}$ using column convention: the action of a $\phi \in
Aut\left( {\mathcal{E}}\right) $ on the $i$-th basis element of ${\mathcal{E}%
}$ is given by $i$-th column of the matrix of $\phi $.

\begin{thm}
(See \cite{Hegazi}.) Any nilpotent evolution algebra of dimension $\leq 3$
over an algebraically closed field $\mathbb{F}$ is isomorphic to one of the
following algebras:%
\begin{table}[H] \centering%
$%
\begin{tabular}{|c|c|c|c|c|c|}
\hline
${\mathcal{E}}$ & Multiplication Table & ${\mathcal{B}}\left( {\mathcal{%
E\times E}},\mathbb{F}\right) $ & ${\mathcal{H}}\left( {\mathcal{E\times E}},%
\mathbb{F}\right) $ & $ann\left( {\mathcal{E}}\right) $ & $Aut\left( {%
\mathcal{E}}\right) $ \\ \hline
${\mathcal{E}}_{1,1}$ & \multicolumn{1}{|l|}{} & \multicolumn{1}{|l|}{} & 
\multicolumn{1}{|l|}{$Sym\left( {\mathcal{E}}_{1,1}\right) $} & ${\mathcal{E}%
}_{1,1}$ & $GL\left( {\mathcal{E}}_{1,1}\right) $ \\ \hline
${\mathcal{E}}_{2,1}$ & \multicolumn{1}{|l|}{} & \multicolumn{1}{|l|}{} & 
\multicolumn{1}{|l|}{$Sym\left( {\mathcal{E}}_{2,1}\right) $} & ${\mathcal{E}%
}_{2,1}$ & $GL\left( {\mathcal{E}}_{2,1}\right) $ \\ \hline
${\mathcal{E}}_{2,2}$ & \multicolumn{1}{|l|}{$e_{1}^{2}=e_{2}$} & 
\multicolumn{1}{|l|}{$\left\langle \delta _{e_{1},e_{1}}\right\rangle $} & 
\multicolumn{1}{|l|}{$\left\langle \left[ \delta _{e_{2},e_{2}}\right]
\right\rangle $} & $\left\langle e_{2}\right\rangle $ & $\phi =%
\begin{bmatrix}
a_{11} & 0 \\ 
a_{21} & a_{11}^{2}%
\end{bmatrix}%
$ \\ \hline
${\mathcal{E}}_{3,1}$ & \multicolumn{1}{|l|}{} & \multicolumn{1}{|l|}{} & 
\multicolumn{1}{|l|}{$Sym\left( {\mathcal{E}}_{3,1}\right) $} & ${\mathcal{E}%
}_{3,1}$ & $GL\left( {\mathcal{E}}_{3,1}\right) $ \\ \hline
${\mathcal{E}}_{3,2}$ & \multicolumn{1}{|l|}{$e_{1}^{2}=e_{2}$} & 
\multicolumn{1}{|l|}{$\left\langle \delta _{e_{1},e_{1}}\right\rangle $} & 
\multicolumn{1}{|l|}{$\left\langle \left[ \delta _{e_{2},e_{2}}\right] ,%
\left[ \delta _{e_{3},e_{3}}\right] \right\rangle $} & $\left\langle
e_{2},e_{3}\right\rangle $ & $\phi =%
\begin{bmatrix}
a_{11} & 0 & 0 \\ 
a_{21} & a_{11}^{2} & a_{23} \\ 
a_{31} & 0 & a_{33}%
\end{bmatrix}%
$ \\ \hline
${\mathcal{E}}_{3,3}$ & \multicolumn{1}{|l|}{$%
e_{1}^{2}=e_{3},e_{2}^{2}=e_{3} $} & \multicolumn{1}{|l|}{$\left\langle
\delta _{e_{1},e_{1}}+\delta _{e_{2},e_{2}}\right\rangle $} & 
\multicolumn{1}{|l|}{$\left\langle \left[ \delta _{e_{1},e_{1}}\right] ,%
\left[ \delta _{e_{3},e_{3}}\right] \right\rangle $} & $\left\langle
e_{3}\right\rangle $ & $%
\begin{array}{l}
\phi =%
\begin{bmatrix}
a_{11} & a_{12} & 0 \\ 
a_{21} & a_{22} & 0 \\ 
a_{31} & a_{32} & a_{33}%
\end{bmatrix}%
, \\ 
a_{11}^{2}+a_{21}^{2}=a_{33}, \\ 
a_{12}^{2}+a_{22}^{2}=a_{33}, \\ 
a_{11}a_{12}+a_{21}a_{22}=0.%
\end{array}%
$ \\ \hline
${\mathcal{E}}_{3,4}$ & \multicolumn{1}{|l|}{$%
e_{1}^{2}=e_{2},e_{2}^{2}=e_{3} $} & \multicolumn{1}{|l|}{$\left\langle
\delta _{e_{1},e_{1}},\delta _{e_{2},e_{2}}\right\rangle $} & 
\multicolumn{1}{|l|}{$\left\langle \left[ \delta _{e_{3},e_{3}}\right]
\right\rangle $} & $\left\langle e_{3}\right\rangle $ & $\phi =%
\begin{bmatrix}
a_{11} & 0 & 0 \\ 
0 & a_{11}^{2} & 0 \\ 
a_{31} & 0 & a_{11}^{4}%
\end{bmatrix}%
$ \\ \hline
\end{tabular}%
$%
\caption{Nilpotent evolution algebras of dimension up to three.}%
\end{table}%
\end{thm}

Let us now classify 4-dimensional nilpotent evolution algebras. For this, we
have the following subsections.

\subsection{Nilpotent evolution algebras with annihilator components}

By Definition \ref{annihilator component}, we get the algebras ${\mathcal{E}}%
_{4,i}={\mathcal{E}}_{3,i}\oplus {\mathcal{E}}_{1,1}$, $i=1,\ldots ,4$.

\subsection{\textbf{1-dimensional annihilator extensions of} ${\mathcal{E}}%
_{3,1}$}

Let $\left[ \theta \right] =\alpha _{1}\left[ \delta _{e_{1},e_{1}}\right]
+\alpha _{2}\left[ \delta _{e_{2},e_{2}}\right] +\alpha _{3}\left[ \delta
_{e_{3},e_{3}}\right] \in {\mathcal{H}}\left( {\mathcal{E}}_{3,1}{\mathcal{%
\times E}}_{3,1},\mathbb{F}\right) $ such that $\theta ^{\bot }\cap
ann\left( {\mathcal{E}}_{3,1}\right) =0$. Then $\alpha _{1}\alpha _{2}\alpha
_{3}\neq 0$. Let $\phi $ be the following automorphism%
\begin{equation*}
\phi =%
\begin{bmatrix}
\alpha _{1}^{-\frac{1}{2}} & 0 & 0 \\ 
0 & \alpha _{2}^{-\frac{1}{2}} & 0 \\ 
0 & 0 & \alpha _{3}^{-\frac{1}{2}}%
\end{bmatrix}%
.
\end{equation*}%
Then $\left[ \phi \theta \right] =\left[ \delta _{e_{1},e_{1}}\right] +\left[
\delta _{e_{2},e_{2}}\right] +\left[ \delta _{e_{3},e_{3}}\right] $. So, we
may assume $\theta =\delta _{e_{1},e_{1}}+\delta _{e_{2},e_{2}}+\delta
_{e_{3},e_{3}}$. Hence we get the algebra ${\mathcal{E}}%
_{4,5}:e_{1}^{2}=e_{2}^{2}=e_{3}^{2}=e_{4}.$

\subsection{\textbf{1-dimensional annihilator extensions of} ${\mathcal{E}}%
_{3,2}$}

Let $\left[ \theta \right] =\alpha _{1}\left[ \delta _{e_{2},e_{2}}\right]
+\alpha _{2}\left[ \delta _{e_{3},e_{3}}\right] \in {\mathcal{H}}\left( {%
\mathcal{E}}_{3,2}{\mathcal{\times E}}_{3,2},\mathbb{F}\right) $ with $%
\alpha _{1}\alpha _{2}\neq 0$. Let $\phi $ be the following automorphism%
\begin{equation*}
\phi =%
\begin{bmatrix}
\alpha _{1}^{-\frac{1}{4}} & 0 & 0 \\ 
0 & \alpha _{1}^{-\frac{1}{2}} & 0 \\ 
0 & 0 & \alpha _{2}^{-\frac{1}{2}}%
\end{bmatrix}%
.
\end{equation*}%
Then $\left[ \phi \theta \right] =\left[ \delta _{e_{2},e_{2}}\right] +\left[
\delta _{e_{3},e_{3}}\right] $. So, we may assume $\theta =\delta
_{e_{2},e_{2}}+\delta _{e_{3},e_{3}}$. Hence we get the algebra ${\mathcal{E}%
}_{4,6}:e_{1}^{2}=e_{2},e_{2}^{2}=e_{3}^{2}=e_{4}.$

\subsection{\textbf{1-dimensional annihilator extensions of} ${\mathcal{E}}%
_{3,3}$}

Let $\left[ \theta \right] =\alpha _{1}\left[ \delta _{e_{1},e_{1}}\right]
+\alpha _{2}\left[ \delta _{e_{3},e_{3}}\right] \in {\mathcal{H}}\left( {%
\mathcal{E}}_{3,3}{\mathcal{\times E}}_{3,3},\mathbb{F}\right) $ with $%
\alpha _{2}\neq 0$. If $\alpha _{1}=0$, then $\left[ \theta \right] =\alpha
_{2}\left[ \delta _{e_{3},e_{3}}\right] $. Since $\alpha _{2}\left[ \delta
_{e_{3},e_{3}}\right] $ spans the same subspace as $\left[ \delta
_{e_{3},e_{3}}\right] $, we may assume that $\theta =\delta _{e_{3},e_{3}}$.
Thus we get the algebra ${\mathcal{E}}%
_{4,7}:e_{1}^{2}=e_{3},e_{2}^{2}=e_{3},e_{3}^{2}=e_{4}$. If $\alpha _{1}\neq
0$, we choose $\phi $ to be the following automorphism%
\begin{equation*}
\phi =%
\begin{bmatrix}
\alpha _{1}^{\frac{1}{2}}\alpha _{2}^{-\frac{1}{2}} & 0 & 0 \\ 
0 & \alpha _{1}^{\frac{1}{2}}\alpha _{2}^{-\frac{1}{2}} & 0 \\ 
0 & 0 & \alpha _{1}\alpha _{2}^{-1}%
\end{bmatrix}%
.
\end{equation*}%
Then $\left[ \phi \theta \right] =\alpha _{1}^{2}\alpha _{2}^{-1}\left( %
\left[ \delta _{e_{2},e_{2}}\right] +\left[ \delta _{e_{3},e_{3}}\right]
\right) $. So, we may assume $\theta =\delta _{e_{2},e_{2}}+\delta
_{e_{3},e_{3}}$. Hence we get the algebra ${\mathcal{E}}%
_{4,8}:e_{1}^{2}=e_{3}+e_{4},e_{2}^{2}=e_{3},e_{3}^{2}=e_{4}$. Further, ${%
\mathcal{E}}_{4,7}$ is not isomorphic to ${\mathcal{E}}_{4,8}$ since $\Psi
\left( \left\langle \left[ \delta _{e_{3},e_{3}}\right] \right\rangle
\right) =\left( 2\right) $ while $\Psi \left( \left\langle \left[ \delta
_{e_{2},e_{2}}\right] +\left[ \delta _{e_{3},e_{3}}\right] \right\rangle
\right) =\left( 1\right) $.

\subsection{\textbf{1-dimensional annihilator extensions of} ${\mathcal{E}}%
_{3,4}$}

As ${\mathcal{H}}\left( {\mathcal{E}}_{3,4}{\mathcal{\times E}}_{3,4},%
\mathbb{F}\right) $ is spanned by $\left\langle \left[ \delta _{e_{3},e_{3}}%
\right] \right\rangle $, we may assume $\theta =\delta _{e_{3},e_{3}}$.
Therefore we get the algebra ${\mathcal{E}}%
_{4,9}:e_{1}^{2}=e_{2},e_{2}^{2}=e_{3},e_{3}^{2}=e_{4}.$

\subsection{\textbf{2-dimensional annihilator extensions of} ${\mathcal{E}}%
_{2,1}$}

Here ${\mathcal{H}}\left( {\mathcal{E}}_{2,1}{\mathcal{\times E}}_{2,1},%
\mathbb{F}\right) =\left\langle \left[ \delta _{e_{1},e_{1}}\right] ,\left[
\delta _{e_{2},e_{2}}\right] \right\rangle $ and therefore we get only one
algebra, namely ${\mathcal{E}}_{4,9}:e_{1}^{2}=e_{3},e_{2}^{2}=e_{4}$.

\begin{thm}
Any 4-dimensional nilpotent evolution algebra over an algebraically closed
field $\mathbb{F}$ is isomorphic to one of the following pairwise
non-isomorphic algebras: 
\begin{table}[H] \centering%
$%
\begin{tabular}{|c|c|c|c|c|}
\hline
${\mathcal{E}}$ & Multiplication Table & ${\mathcal{B}}\left( {\mathcal{%
E\times E}},\mathbb{F}\right) $ & ${\mathcal{H}}\left( {\mathcal{E\times E}},%
\mathbb{F}\right) $ & $ann\left( {\mathcal{E}}\right) $ \\ \hline
${\mathcal{E}}_{4,1}$ & \multicolumn{1}{|l|}{} & \multicolumn{1}{|l|}{} & 
\multicolumn{1}{|l|}{$Sym\left( {\mathcal{E}}_{4,1}\right) $} & ${\mathcal{E}%
}_{4,1}$ \\ \hline
${\mathcal{E}}_{4,2}$ & \multicolumn{1}{|l|}{$e_{1}^{2}=e_{2}.$} & 
\multicolumn{1}{|l|}{$\left\langle \delta _{e_{1},e_{1}}\right\rangle $} & 
\multicolumn{1}{|l|}{$\left\langle \left[ \delta _{e_{2},e_{2}}\right] ,%
\left[ \delta _{e_{3},e_{3}}\right] ,\left[ \delta _{e_{4},e_{4}}\right]
\right\rangle $} & $\left\langle e_{2},e_{3},e_{4}\right\rangle $ \\ \hline
${\mathcal{E}}_{4,3}$ & \multicolumn{1}{|l|}{$%
e_{1}^{2}=e_{3},e_{2}^{2}=e_{3}.$} & \multicolumn{1}{|l|}{$\left\langle
\delta _{e_{1},e_{1}}+\delta _{e_{2},e_{2}}\right\rangle $} & 
\multicolumn{1}{|l|}{$\left\langle \left[ \delta _{e_{1},e_{1}}\right] ,%
\left[ \delta _{e_{3},e_{3}}\right] ,\left[ \delta _{e_{4},e_{4}}\right]
\right\rangle $} & $\left\langle e_{3},e_{4}\right\rangle $ \\ \hline
${\mathcal{E}}_{4,4}$ & \multicolumn{1}{|l|}{$%
e_{1}^{2}=e_{2},e_{2}^{2}=e_{3}.$} & \multicolumn{1}{|l|}{$\left\langle
\delta _{e_{1},e_{1}},\delta _{e_{2},e_{2}}\right\rangle $} & 
\multicolumn{1}{|l|}{$\left\langle \left[ \delta _{e_{3},e_{3}}\right] ,%
\left[ \delta _{e_{4},e_{4}}\right] \right\rangle $} & $\left\langle
e_{3},e_{4}\right\rangle $ \\ \hline
${\mathcal{E}}_{4,5}$ & \multicolumn{1}{|l|}{$%
e_{1}^{2}=e_{4},e_{2}^{2}=e_{4},e_{3}^{2}=e_{4}.$} & \multicolumn{1}{|l|}{$%
\left\langle \delta _{e_{1},e_{1}}+\delta _{e_{2},e_{2}}+\delta
_{e_{3},e_{3}}\right\rangle $} & \multicolumn{1}{|l|}{$\left\langle \left[
\delta _{e_{2},e_{2}}\right] ,\left[ \delta _{e_{3},e_{3}}\right] ,\left[
\delta _{e_{4},e_{4}}\right] \right\rangle $} & $\left\langle
e_{4}\right\rangle $ \\ \hline
${\mathcal{E}}_{4,6}$ & \multicolumn{1}{|l|}{$%
e_{1}^{2}=e_{2},e_{2}^{2}=e_{4},e_{3}^{2}=e_{4}.$} & \multicolumn{1}{|l|}{$%
\left\langle \delta _{e_{1},e_{1}},\delta _{e_{2},e_{2}}+\delta
_{e_{3},e_{3}}\right\rangle $} & \multicolumn{1}{|l|}{$\left\langle \left[
\delta _{e_{3},e_{3}}\right] ,\left[ \delta _{e_{4},e_{4}}\right]
\right\rangle $} & $\left\langle e_{4}\right\rangle $ \\ \hline
${\mathcal{E}}_{4,7}$ & \multicolumn{1}{|l|}{$%
e_{1}^{2}=e_{3},e_{2}^{2}=e_{3},e_{3}^{2}=e_{4}.$} & \multicolumn{1}{|l|}{$%
\left\langle \delta _{e_{1},e_{1}}+\delta _{e_{2},e_{2}},\delta
_{e_{3},e_{3}}\right\rangle $} & \multicolumn{1}{|l|}{$\left\langle \left[
\delta _{e_{2},e_{2}}\right] ,\left[ \delta _{e_{4},e_{4}}\right]
\right\rangle $} & $\left\langle e_{4}\right\rangle $ \\ \hline
${\mathcal{E}}_{4,8}$ & \multicolumn{1}{|c|}{$%
e_{1}^{2}=e_{3}+e_{4},e_{2}^{2}=e_{3},e_{3}^{2}=e_{4}.$} & 
\multicolumn{1}{|l|}{$\left\langle \delta _{e_{1},e_{1}}+\delta
_{e_{2},e_{2}},\delta _{e_{1},e_{1}}+\delta _{e_{3},e_{3}}\right\rangle $} & 
\multicolumn{1}{|l|}{$\left\langle \left[ \delta _{e_{3},e_{3}}\right] ,%
\left[ \delta _{e_{4},e_{4}}\right] \right\rangle $} & $\left\langle
e_{4}\right\rangle $ \\ \hline
${\mathcal{E}}_{4,9}$ & \multicolumn{1}{|l|}{$%
e_{1}^{2}=e_{2},e_{2}^{2}=e_{3},e_{3}^{2}=e_{4}.$} & \multicolumn{1}{|l|}{$%
\left\langle \delta _{e_{1},e_{1}},\delta _{e_{2},e_{2}},\delta
_{e_{3},e_{3}}\right\rangle $} & \multicolumn{1}{|l|}{$\left\langle \left[
\delta _{e_{4},e_{4}}\right] \right\rangle $} & $\left\langle
e_{4}\right\rangle $ \\ \hline
${\mathcal{E}}_{4,10}$ & \multicolumn{1}{|l|}{$%
e_{1}^{2}=e_{3},e_{2}^{2}=e_{4}.$} & \multicolumn{1}{|l|}{$\left\langle
\delta _{e_{1},e_{1}},\delta _{e_{2},e_{2}}\right\rangle $} & 
\multicolumn{1}{|l|}{$\left\langle \left[ \delta _{e_{3},e_{3}}\right] ,%
\left[ \delta _{e_{4},e_{4}}\right] \right\rangle $} & $\left\langle
e_{3},e_{4}\right\rangle $ \\ \hline
\end{tabular}%
$%
\caption{Nilpotent evolution algebras of dimension four.}%
\label{tab1}%
\end{table}%
\end{thm}

\section{Five-dimensional nilpotent evolution algebras}

\label{classification}In this section we give a complete classification of
five-dimension nilpotent evolution algebras over an algebraically closed
field $\mathbb{F}$.

\subsection{Nilpotent evolution algebras with annihilator components}

By Definition \ref{annihilator component}, we get the algebras ${\mathcal{E}}%
_{5,i}={\mathcal{E}}_{4,i}\oplus {\mathcal{E}}_{1,1}$, $i=1,\ldots ,10$.

\subsection{\textbf{1-dimensional annihilator extensions of} ${\mathcal{E}}%
_{4,1}$}

Here $Aut\left( {\mathcal{E}}_{4,1}\right) =GL\left( {\mathcal{E}}%
_{4,1}\right) $. Moreover, ${\mathcal{H}}\left( {\mathcal{E\times E}},%
\mathbb{F}\right) =Sym\left( {\mathcal{E}}_{4,1}\right) $ and $ann\left( {%
\mathcal{E}}_{4,1}\right) ={\mathcal{E}}_{4,1}$. Let $\theta =\underset{i=1}{%
\overset{4}{\sum }}\alpha _{i}\delta _{e_{i},e_{i}}$ with $\alpha _{1}\alpha
_{2}\alpha _{3}\alpha _{4}\neq 0$. Let $\phi $ be the following automorphism%
\begin{equation*}
\phi =%
\begin{bmatrix}
\alpha _{1}^{-\frac{1}{2}} & 0 & 0 & 0 \\ 
0 & \alpha _{2}^{-\frac{1}{2}} & 0 & 0 \\ 
0 & 0 & \alpha _{3}^{-\frac{1}{2}} & 0 \\ 
0 & 0 & 0 & \alpha _{4}^{-\frac{1}{2}}%
\end{bmatrix}%
.
\end{equation*}%
Then $\phi \theta =\underset{i=1}{\overset{4}{\sum }}\delta _{e_{i},e_{i}}$.
So we may assume that $\theta =\underset{i=1}{\overset{4}{\sum }}\delta
_{e_{i},e_{i}}$. Therefore we get the algebra%
\begin{equation*}
{\mathcal{E}}_{5,11}:e_{1}^{2}=e_{2}^{2}=e_{3}^{2}=e_{4}^{2}=e_{5}.
\end{equation*}

\subsection{\textbf{1-dimensional annihilator extensions of} ${\mathcal{E}}%
_{4,2}$}

The automorphism group $Aut\left( {\mathcal{E}}_{4,2}\right) $\ consists of%
\begin{equation*}
\phi =%
\begin{bmatrix}
a_{11} & 0 & 0 & 0 \\ 
a_{21} & a_{11}^{2} & a_{23} & a_{24} \\ 
a_{31} & 0 & a_{33} & a_{34} \\ 
a_{41} & 0 & a_{43} & a_{44}%
\end{bmatrix}%
:a_{11}^{3}\left( a_{33}a_{44}-a_{34}a_{43}\right) \neq 0
\end{equation*}%
Let $\left[ \theta \right] =\alpha _{1}\left[ \delta _{e_{2},e_{2}}\right]
+\alpha _{2}\left[ \delta _{e_{3},e_{3}}\right] +\alpha _{3}\left[ \delta
_{e_{4},e_{4}}\right] $ with $\alpha _{1}\alpha _{2}\alpha _{3}\neq 0$.
Consider the following automorphism%
\begin{equation*}
\phi =%
\begin{bmatrix}
\alpha _{1}^{-\frac{1}{4}} & 0 & 0 & 0 \\ 
0 & \alpha _{1}^{-\frac{1}{2}} & 0 & 0 \\ 
0 & 0 & \alpha _{2}^{-\frac{1}{2}} & 0 \\ 
0 & 0 & 0 & \alpha _{3}^{-\frac{1}{2}}%
\end{bmatrix}%
.
\end{equation*}%
Then $\left[ \phi \theta \right] =\left[ \delta _{e_{2},e_{2}}\right] +\left[
\delta _{e_{3},e_{3}}\right] +\left[ \delta _{e_{4},e_{4}}\right] $. So we
get the algebra 
\begin{equation*}
{\mathcal{E}}_{5,12}:e_{1}^{2}=e_{2},e_{2}^{2}=e_{3}^{2}=e_{4}^{2}=e_{5}.
\end{equation*}

\subsection{\textbf{1-dimensional annihilator extensions of }${\mathcal{E}}%
_{4,3}$}

The automorphism group $Aut\left( {\mathcal{E}}_{4,3}\right) $\ consists of%
\begin{equation*}
\phi =%
\begin{bmatrix}
a_{11} & a_{12} & 0 & 0 \\ 
a_{21} & a_{22} & 0 & 0 \\ 
a_{31} & a_{32} & a_{33} & a_{34} \\ 
a_{41} & a_{42} & 0 & a_{44}%
\end{bmatrix}%
,
\end{equation*}%
such that%
\begin{eqnarray*}
a_{11}a_{12}+a_{22}a_{21} &=&0, \\
a_{11}^{2}+a_{21}^{2}-a_{33} &=&0, \\
a_{12}^{2}+a_{22}^{2}-a_{33} &=&0, \\
a_{33}a_{44}\left( a_{11}a_{22}-a_{12}a_{21}\right) &\neq &0.
\end{eqnarray*}%
Let $\left[ \theta \right] =\alpha _{1}\left[ \delta _{e_{1},e_{1}}\right]
+\alpha _{2}\left[ \delta _{e_{3},e_{3}}\right] +\alpha _{3}\left[ \delta
_{e_{4},e_{4}}\right] $ with $\alpha _{2}\alpha _{3}\neq 0$. Suppose first
that $\alpha _{1}=0$. Consider the following automorphism%
\begin{equation*}
\phi =%
\begin{bmatrix}
\alpha _{2}^{-\frac{1}{4}} & 0 & 0 & 0 \\ 
0 & \alpha _{2}^{-\frac{1}{4}} & 0 & 0 \\ 
0 & 0 & \alpha _{2}^{-\frac{1}{2}} & 0 \\ 
0 & 0 & 0 & \alpha _{3}^{-\frac{1}{2}}%
\end{bmatrix}%
.
\end{equation*}%
Then $\left[ \phi \theta \right] =\left[ \delta _{e_{3},e_{3}}\right] +\left[
\delta _{e_{4},e_{4}}\right] $. So we get the algebra%
\begin{equation*}
{\mathcal{E}}_{5,13}:e_{1}^{2}=e_{2}^{2}=e_{3},e_{3}^{2}=e_{4}^{2}=e_{5}.
\end{equation*}%
Assume now that $\alpha _{1}=0$. Let $\phi $ be the following automorphism%
\begin{equation*}
\phi =%
\begin{bmatrix}
\alpha _{1}^{\frac{1}{2}}\alpha _{2}^{-\frac{1}{2}} & 0 & 0 & 0 \\ 
0 & \alpha _{1}^{\frac{1}{2}}\alpha _{2}^{-\frac{1}{2}} & 0 & 0 \\ 
0 & 0 & \alpha _{1}\alpha _{2}^{-1} & 0 \\ 
0 & 0 & 0 & \alpha _{1}\alpha _{2}^{-\frac{1}{2}}\alpha _{3}^{-\frac{1}{2}}%
\end{bmatrix}%
.
\end{equation*}%
Then $\left[ \phi \theta \right] =\alpha _{1}^{2}\alpha _{2}^{-1}\left( %
\left[ \delta _{e_{1},e_{1}}\right] +\left[ \delta _{e_{3},e_{3}}\right] +%
\left[ \delta _{e_{4},e_{4}}\right] \right) $. So we may assume that $\theta
=\delta _{e_{1},e_{1}}+\delta _{e_{3},e_{3}}+\delta _{e_{4},e_{4}}$. Hence
we get the algebra%
\begin{equation*}
{\mathcal{E}}%
_{5,14}:e_{1}^{2}=e_{3}+e_{5},e_{2}^{2}=e_{3},e_{3}^{2}=e_{4}^{2}=e_{5}.
\end{equation*}%
Further, ${\mathcal{E}}_{5,13},$ ${\mathcal{E}}_{5,14}$ are non-isomorphic
since $\Psi \left( \left\langle \left[ \delta _{e_{3},e_{3}}\right] +\left[
\delta _{e_{4},e_{4}}\right] \right\rangle \right) =\left( 2\right) $ while $%
\Psi \left( \left\langle \left[ \delta _{e_{1},e_{1}}\right] +\left[ \delta
_{e_{3},e_{3}}\right] +\left[ \delta _{e_{4},e_{4}}\right] \right\rangle
\right) =\left( 1\right) $.

\subsection{\textbf{1-dimensional annihilator extensions of }${\mathcal{E}}%
_{4,4}$}

The automorphism group $Aut\left( {\mathcal{E}}_{4,4}\right) $\ consists of%
\begin{equation*}
\phi =%
\begin{bmatrix}
a_{11} & 0 & 0 & 0 \\ 
0 & a_{11}^{2} & 0 & 0 \\ 
a_{31} & 0 & a_{11}^{4} & 0 \\ 
a_{41} & 0 & 0 & a_{44}%
\end{bmatrix}%
:a_{11}a_{44}\neq 0.
\end{equation*}%
Let $\left[ \theta \right] =\alpha _{1}\left[ \delta _{e_{3},e_{3}}\right]
+\alpha _{2}\left[ \delta _{e_{4},e_{4}}\right] \ $with $\alpha _{1}\alpha
_{2}\neq 0$. Consider the following automorphism%
\begin{equation*}
\phi =%
\begin{bmatrix}
\alpha _{1}^{-\frac{1}{8}} & 0 & 0 & 0 \\ 
0 & \alpha _{1}^{-\frac{1}{4}} & 0 & 0 \\ 
0 & 0 & \alpha _{1}^{-\frac{1}{2}} & 0 \\ 
0 & 0 & 0 & \alpha _{2}^{-\frac{1}{2}}%
\end{bmatrix}%
.
\end{equation*}%
Then $\left[ \phi \theta \right] =\left[ \delta _{e_{3},e_{3}}\right] +\left[
\delta _{e_{4},e_{4}}\right] $. Hence we get the algebra 
\begin{equation*}
{\mathcal{E}}%
_{5,15}:e_{1}^{2}=e_{2},e_{2}^{2}=e_{3},e_{3}^{2}=e_{4}^{2}=e_{5}.
\end{equation*}

\subsection{\textbf{1-dimensional annihilator extensions of }${\mathcal{E}}%
_{4,5}$}

\label{illus}The automorphism group $Aut\left( {\mathcal{E}}_{4,5}\right) $\
consists of%
\begin{equation*}
\phi =%
\begin{bmatrix}
a_{11} & a_{12} & a_{13} & 0 \\ 
a_{21} & a_{22} & a_{23} & 0 \\ 
a_{31} & a_{32} & a_{33} & 0 \\ 
a_{41} & a_{42} & a_{43} & a_{44}%
\end{bmatrix}%
\end{equation*}%
such that%
\begin{eqnarray*}
a_{11}a_{13}+a_{21}a_{23}+a_{31}a_{33} &=&0, \\
a_{12}a_{13}+a_{22}a_{23}+a_{32}a_{33} &=&0, \\
a_{11}a_{12}+a_{21}a_{22}+a_{31}a_{32} &=&0, \\
a_{11}^{2}+a_{21}^{2}+a_{31}^{2}-a_{44} &=&0, \\
a_{12}^{2}+a_{22}^{2}+a_{32}^{2}-a_{44} &=&0, \\
a_{13}^{2}+a_{23}^{2}+a_{33}^{2}-a_{44} &=&0 \\
\det \phi &\neq &0.
\end{eqnarray*}%
Let $\left[ \theta \right] =\alpha _{1}\left[ \delta _{e_{2},e_{2}}\right]
+\alpha _{2}\left[ \delta _{e_{3},e_{3}}\right] +\alpha _{3}\left[ \delta
_{e_{4},e_{4}}\right] $ with $\alpha _{3}\neq 0$. Let us consider the
following cases:

\begin{enumerate}
\item $\alpha _{1}=\alpha _{2}=0$. Then we may assume that $\theta =\delta
_{e_{4},e_{4}}$. So we get the algebra%
\begin{equation*}
{\mathcal{E}}_{5,16}:e_{1}^{2}=e_{2}^{2}=e_{3}^{2}=e_{4},e_{4}^{2}=e_{5}.
\end{equation*}

\item $\alpha _{1}\neq 0,\alpha _{2}=0$. Let $\phi $ be the following
automorphism%
\begin{equation*}
\phi =%
\begin{bmatrix}
\alpha _{1}^{\frac{1}{2}}\alpha _{3}^{-\frac{1}{2}} & 0 & 0 & 0 \\ 
0 & \alpha _{1}^{\frac{1}{2}}\alpha _{3}^{-\frac{1}{2}} & 0 & 0 \\ 
0 & 0 & \alpha _{1}^{\frac{1}{2}}\alpha _{3}^{-\frac{1}{2}} & 0 \\ 
0 & 0 & 0 & \alpha _{1}\alpha _{3}^{-1}%
\end{bmatrix}%
.
\end{equation*}%
Then $\left[ \phi \theta \right] =\alpha _{1}^{2}\alpha _{3}^{-1}\left( %
\left[ \delta _{e_{2},e_{2}}\right] +\left[ \delta _{e_{4},e_{4}}\right]
\right) $. So we may assume that $\theta =\delta _{e_{2},e_{2}}+\delta
_{e_{4},e_{4}}$. Therefore we get the algebra%
\begin{equation*}
{\mathcal{E}}%
_{5,17}:e_{1}^{2}=e_{3}^{2}=e_{4},e_{2}^{2}=e_{4}+e_{5},e_{4}^{2}=e_{5}.
\end{equation*}

\item $\alpha _{1}=0,\alpha _{2}\neq 0$. Let $\phi $ be the following
automorphism%
\begin{equation*}
\phi =%
\begin{bmatrix}
\alpha _{2}^{\frac{1}{2}}\alpha _{3}^{-\frac{1}{2}} & 0 & 0 & 0 \\ 
0 & 0 & \alpha _{2}^{\frac{1}{2}}\alpha _{3}^{-\frac{1}{2}} & 0 \\ 
0 & \alpha _{2}^{\frac{1}{2}}\alpha _{3}^{-\frac{1}{2}} & 0 & 0 \\ 
0 & 0 & 0 & \alpha _{2}\alpha _{3}^{-1}%
\end{bmatrix}%
.
\end{equation*}%
Then $\left[ \phi \theta \right] =\alpha _{2}^{2}\alpha _{3}^{-1}\left( %
\left[ \delta _{e_{2},e_{2}}\right] +\left[ \delta _{e_{4},e_{4}}\right]
\right) $. So, we again get the algebra ${\mathcal{E}}_{5,17}$.

\item $\alpha _{1}=\alpha _{2},\alpha _{1}^{2}\neq 0$. Let $\phi $ be the
following automorphism 
\begin{equation*}
\phi =%
\begin{bmatrix}
0 & \sqrt{-\alpha _{1}\alpha _{3}^{-1}} & 0 & 0 \\ 
0 & 0 & \sqrt{-\alpha _{1}\alpha _{3}^{-1}} & 0 \\ 
\sqrt{-\alpha _{1}\alpha _{3}^{-1}} & 0 & 0 & 0 \\ 
0 & 0 & 0 & -\alpha _{1}\alpha _{3}^{-1}%
\end{bmatrix}%
.
\end{equation*}%
Then $\left[ \phi \theta \right] =-\alpha _{1}^{2}\alpha _{3}^{-1}\left[
\delta _{e_{1},e_{1}}+\delta _{e_{3},e_{3}}\right] +\alpha _{1}^{2}\alpha
_{3}^{-1}\left[ \delta _{e_{4},e_{4}}\right] =\alpha _{1}^{2}\alpha
_{3}^{-1}\left( \left[ \delta _{e_{2},e_{2}}\right] +\left[ \delta
_{e_{4},e_{4}}\right] \right) $. So we may assume that $\theta =\delta
_{e_{2},e_{2}}+\delta _{e_{4},e_{4}}$. So, we again get the algebra ${%
\mathcal{E}}_{5,17}$.

\item $\alpha _{1}\alpha _{2}\left( \alpha _{1}-\alpha _{2}\right) \neq 0$.
Let $\phi $ be the following automorphism%
\begin{equation*}
\phi =%
\begin{bmatrix}
\alpha _{1}^{\frac{1}{2}}\alpha _{3}^{-\frac{1}{2}} & 0 & 0 & 0 \\ 
0 & \alpha _{1}^{\frac{1}{2}}\alpha _{3}^{-\frac{1}{2}} & 0 & 0 \\ 
0 & 0 & \alpha _{1}^{\frac{1}{2}}\alpha _{3}^{-\frac{1}{2}} & 0 \\ 
0 & 0 & 0 & \alpha _{1}\alpha _{3}^{-1}%
\end{bmatrix}%
.
\end{equation*}%
Then $\left[ \phi \theta \right] =\alpha _{1}^{2}\alpha _{3}^{-1}\left( %
\left[ \delta _{e_{2},e_{2}}\right] +\alpha _{1}^{-1}\alpha _{2}\left[
\delta _{e_{3},e_{3}}\right] +\left[ \delta _{e_{4},e_{4}}\right] \right) $.
So we may assume that $\theta =\delta _{e_{2},e_{2}}+\alpha _{1}^{-1}\alpha
_{2}\delta _{e_{3},e_{3}}+\delta _{e_{4},e_{4}}$. Set $\alpha =\alpha
_{1}^{-1}\alpha _{2}$. Then $\alpha \neq 0,1$. So we get the algebras%
\begin{equation*}
{\mathcal{E}}_{5,18}^{\alpha \in \mathbb{F}^{\ast }\mathbb{-}\left\{
1\right\} }:e_{1}^{2}=e_{4},e_{2}^{2}=e_{4}+e_{5},e_{3}^{2}=e_{4}+\alpha
e_{5},e_{4}^{2}=e_{5}.
\end{equation*}
\end{enumerate}

Since $\Psi \left( \left\langle \left[ \delta _{e_{2},e_{2}}\right] +\left[
\delta _{e_{4},e_{4}}\right] \right\rangle \right) =\left( 2\right) $ and $%
\Psi \left( \left\langle \left[ \delta _{e_{2},e_{2}}\right] +\alpha \left[
\delta _{e_{3},e_{3}}\right] +\left[ \delta _{e_{4},e_{4}}\right]
\right\rangle \right) =\left( 1\right) $, ${\mathcal{E}}_{5,17}$ is
non-isomorphic to ${\mathcal{E}}_{5,18}^{\alpha }$. Further, we claim that ${%
\mathcal{E}}_{5,18}^{\alpha }$ is isomorphic to ${\mathcal{E}}_{5,18}^{\beta
}$ if and only if $\beta =\alpha ,\alpha ^{-1},1-\alpha ,\left( 1-\alpha
\right) ^{-1},\alpha \left( \alpha -1\right) ^{-1},\alpha ^{-1}\left( \alpha
-1\right) $. To see this, suppose first that ${\mathcal{E}}_{5,18}^{\alpha
}\cong {\mathcal{E}}_{5,18}^{\beta }$. Then there exists a $\phi \in
Aut\left( {\mathcal{E}}_{4,5}\right) $ such that $\phi \left( \left[ \delta
_{e_{2},e_{2}}\right] +\alpha \left[ \delta _{e_{3},e_{3}}\right] +\left[
\delta _{e_{4},e_{4}}\right] \right) =\lambda \left( \left[ \delta
_{e_{2},e_{2}}\right] +\beta \left[ \delta _{e_{3},e_{3}}\right] +\left[
\delta _{e_{4},e_{4}}\right] \right) $ for some $\lambda \in \mathbb{F}%
^{\ast }$. This the amounts to the following polynomial equations:%
\begin{eqnarray*}
a_{22}^{2}+\alpha a_{32}^{2}+a_{42}^{2}-\left( a_{21}^{2}+\alpha
a_{31}^{2}+a_{41}^{2}\right) &=&\lambda , \\
a_{23}^{2}+\alpha a_{33}^{2}+a_{43}^{2}-\left( a_{21}^{2}+\alpha
a_{31}^{2}+a_{41}^{2}\right) &=&\lambda \beta , \\
a_{44}^{2} &=&\lambda , \\
a_{21}a_{22}+a_{41}a_{42}+\alpha a_{31}a_{32} &=&0, \\
a_{21}a_{23}+a_{41}a_{43}+\alpha a_{31}a_{33} &=&0, \\
a_{22}a_{23}+a_{42}a_{43}+\alpha a_{32}a_{33} &=&0, \\
a_{41}a_{44} &=&0, \\
a_{42}a_{44} &=&0, \\
a_{43}a_{44} &=&0,
\end{eqnarray*}%
In addition to these equations we have%
\begin{eqnarray*}
a_{11}a_{13}+a_{21}a_{23}+a_{31}a_{33} &=&0, \\
a_{12}a_{13}+a_{22}a_{23}+a_{32}a_{33} &=&0, \\
a_{11}a_{12}+a_{21}a_{22}+a_{31}a_{32} &=&0, \\
a_{11}^{2}+a_{21}^{2}+a_{31}^{2}-a_{44} &=&0, \\
a_{12}^{2}+a_{22}^{2}+a_{32}^{2}-a_{44} &=&0, \\
a_{13}^{2}+a_{23}^{2}+a_{33}^{2}-a_{44} &=&0, \\
\det \phi &\neq &0,
\end{eqnarray*}%
which ensure that $\phi \in Aut\left( {\mathcal{E}}_{4,5}\right) $. To solve
these equations, it is convenient to compute the Gr\H{o}bner basis of the
ideal generated by all of the previous polynomial equations to get an
equivalent set of equations which may be easier to solve. This way, we get
(among others) the following equation:

$\alpha ^{6}\beta ^{4}-2\alpha ^{6}\beta ^{3}+\alpha ^{6}\beta ^{2}-3\alpha
^{5}\beta ^{4}+6\alpha ^{5}\beta ^{3}-3\alpha ^{5}\allowbreak \beta
^{2}-\alpha ^{4}\beta ^{6}+3\alpha ^{4}\beta ^{5}-5\alpha ^{4}\beta
^{3}+3\alpha ^{4}\beta -\alpha ^{4}+2\alpha ^{3}\allowbreak \beta
^{6}-6\alpha ^{3}\beta ^{5}+5\alpha ^{3}\beta ^{4}+5\alpha ^{3}\beta
^{2}-6\alpha ^{3}\beta +2\alpha ^{3}-\alpha ^{2}\allowbreak \beta
^{6}+3\alpha ^{2}\beta ^{5}-5\alpha ^{2}\beta ^{3}+3\alpha ^{2}\beta -\alpha
^{2}-3\alpha \beta ^{4}+6\alpha \beta ^{3}-\allowbreak 3\alpha \beta
^{2}+\beta ^{4}-2\beta ^{3}+\beta ^{2}=0$. This is equivalent to $\left(
\left( 1-\alpha \right) \beta -1\right) \left( \beta -\left( 1-\alpha
\right) \right) \left( \alpha \beta -1\right) \left( \beta -\alpha \right)
\left( \alpha \beta -\left( \alpha -1\right) \right) \left( \beta \left(
\alpha -1\right) -\alpha \right) =0$, as claimed. Conversely, let $\beta
=\alpha ,\alpha ^{-1},1-\alpha ,\left( 1-\alpha \right) ^{-1},\alpha \left(
\alpha -1\right) ^{-1},\alpha ^{-1}\left( \alpha -1\right) $. Consider the
following automorphisms%
\begin{eqnarray*}
\phi _{1} &=&%
\begin{bmatrix}
-\alpha ^{\frac{1}{2}} & 0 & 0 & 0 \\ 
0 & 0 & \alpha ^{\frac{1}{2}} & 0 \\ 
0 & \alpha ^{\frac{1}{2}} & 0 & 0 \\ 
0 & 0 & 0 & \alpha 
\end{bmatrix}%
,\phi _{2}=%
\begin{bmatrix}
0 & \sqrt{-1} & 0 & 0 \\ 
\sqrt{-1} & 0 & 0 & 0 \\ 
0 & 0 & \sqrt{-1} & 0 \\ 
0 & 0 & 0 & -1%
\end{bmatrix}%
,\phi _{3}=%
\begin{bmatrix}
0 & 0 & \left( \alpha -1\right) ^{\frac{1}{2}} & 0 \\ 
\frac{\left( \alpha -1\right) ^{3}}{\sqrt{\left( \alpha -1\right) ^{5}}} & 0
& 0 & 0 \\ 
0 & \left( \alpha -1\right) ^{\frac{1}{2}} & 0 & 0 \\ 
0 & 0 & 0 & \alpha -1%
\end{bmatrix}%
, \\
\phi _{4} &=&%
\begin{bmatrix}
0 & 0 & \left( 1-\alpha \right) ^{\frac{1}{2}} & 0 \\ 
0 & \left( 1-\alpha \right) ^{\frac{1}{2}} & 0 & 0 \\ 
\frac{\left( \alpha -1\right) ^{3}}{\sqrt{-\left( \alpha -1\right) ^{5}}} & 0
& 0 & 0 \\ 
0 & 0 & 0 & \alpha -1%
\end{bmatrix}%
,\phi _{5}=%
\begin{bmatrix}
0 & \sqrt{-\alpha } & 0 & 0 \\ 
0 & 0 & \sqrt{-\alpha } & 0 \\ 
-\sqrt{-\alpha } & 0 & 0 & 0 \\ 
0 & 0 & 0 & -\alpha 
\end{bmatrix}%
.
\end{eqnarray*}%
Then%
\begin{eqnarray*}
\phi _{1}\left( \left[ \delta _{e_{2},e_{2}}\right] +\alpha \left[ \delta
_{e_{3},e_{3}}\right] +\left[ \delta _{e_{4},e_{4}}\right] \right) 
&=&\alpha ^{2}\left( \left[ \delta _{e_{2},e_{2}}\right] +\alpha ^{-1}\left[
\delta _{e_{3},e_{3}}\right] +\left[ \delta _{e_{4},e_{4}}\right] \right) ,
\\
\phi _{2}\left( \left[ \delta _{e_{2},e_{2}}\right] +\alpha \left[ \delta
_{e_{3},e_{3}}\right] +\left[ \delta _{e_{4},e_{4}}\right] \right)  &=&\left[
\delta _{e_{2},e_{2}}\right] +\left( 1-\alpha \right) \left[ \delta
_{e_{3},e_{3}}\right] +\left[ \delta _{e_{4},e_{4}}\right] , \\
\phi _{3}\left( \left[ \delta _{e_{2},e_{2}}\right] +\alpha \left[ \delta
_{e_{3},e_{3}}\right] +\left[ \delta _{e_{4},e_{4}}\right] \right) 
&=&\left( \alpha -1\right) ^{2}\left( \left[ \delta _{e_{2},e_{2}}\right]
+\left( 1-\alpha \right) ^{-1}\left[ \delta _{e_{3},e_{3}}\right] +\left[
\delta _{e_{4},e_{4}}\right] \right) , \\
\phi _{4}\left( \left[ \delta _{e_{2},e_{2}}\right] +\alpha \left[ \delta
_{e_{3},e_{3}}\right] +\left[ \delta _{e_{4},e_{4}}\right] \right) 
&=&\left( \alpha -1\right) ^{2}\left( \left[ \delta _{e_{2},e_{2}}\right]
+\alpha \left( \alpha -1\right) ^{-1}\left[ \delta _{e_{3},e_{3}}\right] +%
\left[ \delta _{e_{4},e_{4}}\right] \right) , \\
\phi _{5}\left( \left[ \delta _{e_{2},e_{2}}\right] +\alpha \left[ \delta
_{e_{3},e_{3}}\right] +\left[ \delta _{e_{4},e_{4}}\right] \right) 
&=&\alpha ^{2}\left( \left[ \delta _{e_{2},e_{2}}\right] +\alpha ^{-1}\left(
\alpha -1\right) \left[ \delta _{e_{3},e_{3}}\right] +\left[ \delta
_{e_{4},e_{4}}\right] \right) .
\end{eqnarray*}%
Therefore, ${\mathcal{E}}_{5,18}^{\alpha }\cong {\mathcal{E}}_{5,18}^{\alpha
^{-1}}\cong {\mathcal{E}}_{5,18}^{1-\alpha }\cong {\mathcal{E}}%
_{5,18}^{\left( 1-\alpha \right) ^{-1}}\cong {\mathcal{E}}_{5,18}^{\alpha
\left( \alpha -1\right) ^{-1}}\cong {\mathcal{E}}_{5,18}^{\alpha ^{-1}\left(
\alpha -1\right) }$. This completes the proof of the claim.

\subsection{\textbf{1-dimensional annihilator extensions of }${\mathcal{E}}%
_{4,6}$}

The automorphism group $Aut\left( {\mathcal{E}}_{4,6}\right) $\ consists of%
\begin{equation*}
\phi =%
\begin{bmatrix}
a_{11} & 0 & 0 & 0 \\ 
0 & a_{11}^{2} & 0 & 0 \\ 
0 & 0 & a_{33} & 0 \\ 
a_{41} & 0 & a_{43} & a_{11}^{4}%
\end{bmatrix}%
:a_{33}^{2}=a_{11}^{4},a_{11}\neq 0.
\end{equation*}%
Let $\left[ \theta \right] =\alpha _{1}\left[ \delta _{e_{3},e_{3}}\right]
+\alpha _{2}\left[ \delta _{e_{4},e_{4}}\right] $ with $\alpha _{2}\neq 0$.
If $\alpha _{1}=0$, we then may assume that $\theta =\delta _{e_{4},e_{4}}$.
So we get the algebra%
\begin{equation*}
{\mathcal{E}}%
_{5,19}:e_{1}^{2}=e_{2},e_{2}^{2}=e_{3}^{2}=e_{4},e_{4}^{2}=e_{5}.
\end{equation*}%
Otherwise, we choose $\phi $ to be the following automorphism%
\begin{equation*}
\phi =%
\begin{bmatrix}
\sqrt[4]{\alpha _{1}\alpha _{2}^{-1}} & 0 & 0 & 0 \\ 
0 & \sqrt{\alpha _{1}\alpha _{2}^{-1}} & 0 & 0 \\ 
0 & 0 & \sqrt{\alpha _{1}\alpha _{2}^{-1}} & 0 \\ 
0 & 0 & 0 & \alpha _{1}\alpha _{2}^{-1}%
\end{bmatrix}%
.
\end{equation*}%
Then $\left[ \phi \theta \right] =\alpha _{1}^{2}\alpha _{2}^{-1}\left( %
\left[ \delta _{e_{3},e_{3}}\right] +\left[ \delta _{e_{4},e_{4}}\right]
\right) $. So we may assume that $\theta =\delta _{e_{3},e_{3}}+\delta
_{e_{4},e_{4}}$. Hence we get the algebra%
\begin{equation*}
{\mathcal{E}}%
_{5,20}:e_{1}^{2}=e_{2},e_{2}^{2}=e_{4},e_{3}^{2}=e_{4}+e_{5},e_{4}^{2}=e_{5}.
\end{equation*}%
Since $\Psi \left( \left\langle \left[ \delta _{e_{4},e_{4}}\right]
\right\rangle \right) =\left( 3\right) $ and $\Psi \left( \left\langle \left[
\delta _{e_{3},e_{3}}\right] +\left[ \delta _{e_{4},e_{4}}\right]
\right\rangle \right) =\left( 2\right) $, ${\mathcal{E}}_{5,19}$ is
non-isomorphic to ${\mathcal{E}}_{5,20}$.

\subsection{\textbf{1-dimensional annihilator extensions of }${\mathcal{E}}%
_{4,7}$}

The automorphism group $Aut\left( {\mathcal{E}}_{4,7}\right) $\ consists of%
\begin{equation*}
\phi =%
\begin{bmatrix}
a_{11} & a_{12} & 0 & 0 \\ 
a_{21} & a_{22} & 0 & 0 \\ 
0 & 0 & a_{33} & 0 \\ 
a_{41} & a_{42} & 0 & a_{33}^{2}%
\end{bmatrix}%
,
\end{equation*}%
such that%
\begin{eqnarray*}
a_{11}a_{12}+a_{22}a_{21} &=&0, \\
a_{11}^{2}+a_{21}^{2}-a_{33} &=&0, \\
a_{12}^{2}+a_{22}^{2}-a_{33} &=&0, \\
a_{33}\left( a_{11}a_{22}-a_{12}a_{21}\right) &\neq &0.
\end{eqnarray*}%
Let $\left[ \theta \right] =\alpha _{1}\left[ \delta _{e_{2},e_{2}}\right]
+\alpha _{2}\left[ \delta _{e_{4},e_{4}}\right] $ with $\alpha _{2}\neq 0$.
If $\alpha _{1}=0$, we then may assume that $\theta =\delta _{e_{4},e_{4}}$.
So we get the algebra%
\begin{equation*}
{\mathcal{E}}%
_{5,21}:e_{1}^{2}=e_{2}^{2}=e_{3},e_{3}^{2}=e_{4},e_{4}^{2}=e_{5}.
\end{equation*}%
Otherwise, we choose $\phi $ to be the following automorphism%
\begin{equation*}
\phi =%
\begin{bmatrix}
\alpha _{1}^{\frac{1}{6}}\alpha _{2}^{-\frac{1}{6}} & 0 & 0 & 0 \\ 
0 & \alpha _{1}^{\frac{1}{6}}\alpha _{2}^{-\frac{1}{6}} & 0 & 0 \\ 
0 & 0 & \alpha _{1}^{\frac{1}{3}}\alpha _{2}^{-\frac{1}{3}} & 0 \\ 
0 & 0 & 0 & \alpha _{1}^{\frac{2}{3}}\alpha _{2}^{-\frac{2}{3}}%
\end{bmatrix}%
.
\end{equation*}%
Then $\left[ \phi \theta \right] =\alpha _{1}^{\frac{4}{3}}\alpha _{2}^{-%
\frac{1}{3}}\left( \left[ \delta _{e_{2},e_{2}}\right] +\left[ \delta
_{e_{4},e_{4}}\right] \right) $. So we may assume that $\theta =\delta
_{e_{2},e_{2}}+\delta _{e_{4},e_{4}}$. Hence we get the algebra%
\begin{equation*}
{\mathcal{E}}%
_{5,22}:e_{1}^{2}=e_{3},e_{2}^{2}=e_{3}+e_{5},e_{3}^{2}=e_{4},e_{4}^{2}=e_{5}.
\end{equation*}%
Further, ${\mathcal{E}}_{5,21},$ ${\mathcal{E}}_{5,22}$ are non-isomorphic
since $\Psi \left( \left\langle \left[ \delta _{e_{4},e_{4}}\right]
\right\rangle \right) =\left( 3\right) $ while $\Psi \left( \left\langle %
\left[ \delta _{e_{2},e_{2}}\right] +\left[ \delta _{e_{4},e_{4}}\right]
\right\rangle \right) =\left( 2\right) $.

\subsection{\textbf{1-dimensional annihilator extensions of }${\mathcal{E}}%
_{4,8}$}

The automorphism group $Aut\left( {\mathcal{E}}_{4,8}\right) $\ consists of%
\begin{equation*}
\phi _{1}=%
\begin{bmatrix}
a_{11} & 0 & 0 & 0 \\ 
0 & a_{22} & 0 & 0 \\ 
0 & 0 & a_{11}^{2} & 0 \\ 
a_{41} & a_{42} & 0 & a_{11}^{2}%
\end{bmatrix}%
,\phi _{2}=%
\begin{bmatrix}
0 & a_{12} & 0 & 0 \\ 
a_{21} & 0 & 0 & 0 \\ 
0 & 0 & a_{12}^{2} & 0 \\ 
a_{41} & a_{42} & -a_{12}^{4} & a_{12}^{4}%
\end{bmatrix}%
:a_{11}\neq 0,a_{11}^{2}=a_{22}^{2},a_{12}\neq 0,a_{12}^{2}=a_{21}^{2}.
\end{equation*}

Let $\left[ \theta \right] =\alpha _{1}\left[ \delta _{e_{3},e_{3}}\right]
+\alpha _{2}\left[ \delta _{e_{4},e_{4}}\right] $ with $\alpha _{2}\neq 0$.
Then we may assume that $\theta =\alpha \delta _{e_{3},e_{3}}+\delta
_{e_{4},e_{4}}$ where $\alpha \in \mathbb{F}$. So we get the algebras%
\begin{equation*}
{\mathcal{E}}_{5,23}^{\alpha \in \mathbb{F}%
}:e_{1}^{2}=e_{3}+e_{4},e_{2}^{2}=e_{3}+\alpha
e_{5},e_{3}^{2}=e_{4},e_{4}^{2}=e_{5}.
\end{equation*}%
Further, we claim that ${\mathcal{E}}_{5,23}^{\alpha }$ is isomorphic to ${%
\mathcal{E}}_{5,23}^{\beta }$ if and only if $\alpha =\beta $. To see this,
let ${\mathcal{E}}_{5,23}^{\alpha }\cong {\mathcal{E}}_{5,23}^{\beta }$.
Then there exist a $\phi \in Aut\left( {\mathcal{E}}_{4,8}\right) $ such that%
\begin{equation}
\phi \left( \alpha \left[ \delta _{e_{3},e_{3}}\right] +\left[ \delta
_{e_{4},e_{4}}\right] \right) =\lambda \left( \beta \left[ \delta
_{e_{3},e_{3}}\right] +\left[ \delta _{e_{4},e_{4}}\right] \right) ;\lambda
\in \mathbb{F}^{\ast }.  \label{H}
\end{equation}%
Suppose first that $\phi =\phi _{1}$. Then Eq. $\left( \ref{H}\right) $\
amounts to the following polynomial equations:%
\begin{equation*}
\alpha a_{11}^{4}-a_{41}^{2}+a_{42}^{2}=\lambda \beta ,a_{11}^{4}=\lambda
,a_{41}a_{42}=0,a_{11}^{2}a_{41}=0,a_{11}^{2}a_{42}=0.
\end{equation*}%
As $a_{11}\neq 0$ we have $a_{41}=a_{42}=0$. Consequently, we get $\alpha
=\beta $. Assume now that $\phi =\phi _{2}$. Then Eq. $\left( \ref{H}\right) 
$ amounts to the following polynomial equations:%
\begin{equation*}
a_{12}^{4}\left( a_{12}^{4}+\alpha \right) -a_{41}^{2}+a_{42}^{2}=\lambda
\beta ,a_{12}^{8}=\lambda
,a_{41}a_{42}=0,a_{12}^{4}a_{41}=0,a_{12}^{4}a_{41}=0,a_{12}^{4}a_{42}=0,a_{12}^{4}a_{42}=0,a_{12}^{8}=0.
\end{equation*}%
The last equation implies that $a_{12}=0$, which is impossible as $\det \phi
\neq 0$ (i.e., $\phi \notin \mathcal{S}_{\theta }\left( {\mathcal{E}}%
_{4,8}\right) $). Therefore, we have $\alpha =\beta $ when ${\mathcal{E}}%
_{5,23}^{\alpha }\cong {\mathcal{E}}_{5,23}^{\beta }$, as claimed.

\subsection{\textbf{1-dimensional annihilator extensions of }${\mathcal{E}}%
_{4,9}$}

Since ${\mathcal{H}}\left( {\mathcal{E\times E}},\mathbb{F}\right)
=\left\langle \left[ \delta _{e_{4},e_{4}}\right] \right\rangle $, we
therefore get the algebra 
\begin{equation*}
{\mathcal{E}}%
_{4,24}:e_{1}^{2}=e_{2},e_{2}^{2}=e_{3},e_{3}^{2}=e_{4},e_{4}^{2}=e_{5}.
\end{equation*}

\subsection{\textbf{1-dimensional annihilator extensions of }${\mathcal{E}}%
_{4,10}$}

The automorphism group $Aut\left( {\mathcal{E}}_{4,10}\right) $\ consists of%
\begin{equation*}
\phi =%
\begin{bmatrix}
a_{11} & 0 & 0 & 0 \\ 
0 & a_{22} & 0 & 0 \\ 
a_{31} & 0 & a_{11}^{2} & 0 \\ 
a_{41} & a_{42} & 0 & a_{22}^{2}%
\end{bmatrix}%
:a_{11}a_{22}\neq 0.
\end{equation*}%
Let $\left[ \theta \right] =\alpha _{1}\left[ \delta _{e_{3},e_{3}}\right]
+\alpha _{2}\left[ \delta _{e_{4},e_{4}}\right] $ with $\alpha _{1}\alpha
_{2}\neq 0$. Let $\phi $ be the following automorphism%
\begin{equation*}
\phi =%
\begin{bmatrix}
\alpha _{1}^{-\frac{1}{4}} & 0 & 0 & 0 \\ 
0 & \alpha _{2}^{-\frac{1}{4}} & 0 & 0 \\ 
0 & 0 & \alpha _{1}^{-\frac{1}{2}} & 0 \\ 
0 & 0 & 0 & \alpha _{2}^{-\frac{1}{2}}%
\end{bmatrix}%
.
\end{equation*}%
Then $\left[ \phi \theta \right] =\left[ \delta _{e_{3},e_{3}}\right] +\left[
\delta _{e_{4},e_{4}}\right] $. So we may assume that $\theta =\delta
_{e_{3},e_{3}}+\delta _{e_{4},e_{4}}$. Thus we get the algebra%
\begin{equation*}
{\mathcal{E}}%
_{4,25}:e_{1}^{2}=e_{3},e_{2}^{2}=e_{4},e_{3}^{2}=e_{4}^{2}=e_{5}.
\end{equation*}

\subsection{\textbf{2-dimensional annihilator extensions of }${\mathcal{E}}%
_{3,1}$}

Here $Aut\left( {\mathcal{E}}_{4,1}\right) =GL\left( {\mathcal{E}}%
_{3,1}\right) $ and ${\mathcal{H}}\left( {\mathcal{E\times E}},\mathbb{F}%
\right) =Sym\left( {\mathcal{E}}_{3,1}\right) $. Let $\theta _{1},\theta
_{2}\in Sym\left( {\mathcal{E}}_{3,1}\right) $ such that $\theta _{1}^{\bot
}\cap \theta _{2}^{\bot }=0$. Without loss of generality we may assume that $%
\theta _{1}=\delta _{e_{1},e_{1}}+\alpha \delta _{e_{3},e_{3}},\theta
_{2}=\delta _{e_{2},e_{2}}+\beta \delta _{e_{3},e_{3}}$. Therefore, $\Psi
\left( \left\langle \theta _{1},\theta _{2}\right\rangle \right) \in \left\{
\left( 2,1\right) ,\left( 1,1\right) \right\} $. Suppose first that $\Psi
\left( \left\langle \theta _{1},\theta _{2}\right\rangle \right) =\left(
2,1\right) $. Then we may assume that $\theta _{1}=\delta
_{e_{1},e_{1}},\theta _{2}=\delta _{e_{2},e_{2}}+\beta \delta _{e_{3},e_{3}}$%
, where $\beta \neq 0$. Let $\phi $ be the following automorphism%
\begin{equation*}
\phi =%
\begin{bmatrix}
1 & 0 & 0 \\ 
0 & 1 & 0 \\ 
0 & 0 & \beta ^{-\frac{1}{2}}%
\end{bmatrix}%
.
\end{equation*}%
Then $\phi \theta _{1}=\theta _{1},\phi \theta _{2}=\delta
_{e_{2},e_{2}}+\delta _{e_{3},e_{3}}$. So we get the algebra%
\begin{equation*}
{\mathcal{E}}_{4,26}:e_{1}^{2}=e_{4},e_{2}^{2}=e_{3}^{2}=e_{5}.
\end{equation*}%
Assume now that $\Psi \left( \left\langle \theta _{1},\theta
_{2}\right\rangle \right) =\left( 1,1\right) $. Then $\alpha \beta \neq 0$.
Let $\phi $ be the following automorphism%
\begin{equation*}
\phi =%
\begin{bmatrix}
1 & 0 & 0 \\ 
0 & \alpha ^{-\frac{1}{2}}\beta ^{\frac{1}{2}} & 0 \\ 
0 & 0 & \alpha ^{-\frac{1}{2}}%
\end{bmatrix}%
.
\end{equation*}%
Then $\phi \theta _{1}=\delta _{e_{1},e_{1}}+\delta _{e_{3},e_{3}},\phi
\theta _{2}=\alpha ^{-1}\beta \left( \delta _{e_{2},e_{2}}+\delta
_{e_{3},e_{3}}\right) $. So we may assume that $\theta _{1}=\delta
_{e_{1},e_{1}}+\delta _{e_{3},e_{3}},\theta _{2}=\delta
_{e_{2},e_{2}}+\delta _{e_{3},e_{3}}$. Hence we get the algebra%
\begin{equation*}
{\mathcal{E}}_{4,27}:e_{1}^{2}=e_{4},e_{2}^{2}=e_{3}^{2}=e_{4}+e_{5}.
\end{equation*}

\subsection{\textbf{2-dimensional annihilator extensions of }${\mathcal{E}}%
_{3,2}$}

As ${\mathcal{H}}\left( {\mathcal{E}}_{3,2}{\mathcal{\times E}}_{3,2},%
\mathbb{F}\right) =\left\langle \left[ \delta _{e_{2},e_{2}}\right] ,\left[
\delta _{e_{3},e_{3}}\right] \right\rangle $ we get only one algebra, namely%
\begin{equation*}
{\mathcal{E}}_{4,28}:e_{1}^{2}=e_{2},e_{2}^{2}=e_{4},e_{3}^{2}=e_{5}.
\end{equation*}

\subsection{\textbf{2-dimensional annihilator extensions of }${\mathcal{E}}%
_{3,3}$}

As ${\mathcal{H}}\left( {\mathcal{E}}_{3,3}{\mathcal{\times E}}_{3,3},%
\mathbb{F}\right) =\left\langle \left[ \delta _{e_{2},e_{2}}\right] ,\left[
\delta _{e_{3},e_{3}}\right] \right\rangle $ we get only one algebra, namely%
\begin{equation*}
{\mathcal{E}}_{4,29}:e_{1}^{2}=e_{3},e_{2}^{2}=e_{3}+e_{4},e_{3}^{2}=e_{5}.
\end{equation*}

\begin{thm}
Any five-dimensional nilpotent evolution algebra with a natural basis $%
\left\{ e_{1},\ldots e_{5}\right\} $\ over an algebraically closed field $%
\mathbb{F}$ is isomorphic to one of the following algebras:

\begin{itemize}
\item ${\mathcal{E}}_{5,1}:$ all products are zero.

\item ${\mathcal{E}}_{5,2}:e_{1}^{2}=e_{2}.$

\item ${\mathcal{E}}_{5,3}:e_{1}^{2}=e_{3},e_{2}^{2}=e_{3}.$

\item ${\mathcal{E}}_{5,4}:e_{1}^{2}=e_{2},e_{2}^{2}=e_{3}.$

\item ${\mathcal{E}}_{5,5}:e_{1}^{2}=e_{4},e_{2}^{2}=e_{4},e_{3}^{2}=e_{4}.$

\item ${\mathcal{E}}_{5,6}:e_{1}^{2}=e_{2},e_{2}^{2}=e_{4},e_{3}^{2}=e_{4}.$

\item ${\mathcal{E}}_{5,7}:e_{1}^{2}=e_{3},e_{2}^{2}=e_{3},e_{3}^{2}=e_{4}.$

\item ${\mathcal{E}}%
_{5,8}:e_{1}^{2}=e_{3}+e_{4},e_{2}^{2}=e_{3},e_{3}^{2}=e_{4}.$

\item ${\mathcal{E}}_{5,9}:e_{1}^{2}=e_{2},e_{2}^{2}=e_{3},e_{3}^{2}=e_{4}.$

\item ${\mathcal{E}}_{5,10}:e_{1}^{2}=e_{3},e_{2}^{2}=e_{4}.$

\item ${\mathcal{E}}_{5,11}:e_{1}^{2}=e_{2}^{2}=e_{3}^{2}=e_{4}^{2}=e_{5}.$

\item ${\mathcal{E}}%
_{5,12}:e_{1}^{2}=e_{2},e_{2}^{2}=e_{3}^{2}=e_{4}^{2}=e_{5}.$

\item ${\mathcal{E}}%
_{5,13}:e_{1}^{2}=e_{2}^{2}=e_{3},e_{3}^{2}=e_{4}^{2}=e_{5}.$

\item ${\mathcal{E}}%
_{5,14}:e_{1}^{2}=e_{3}+e_{5},e_{2}^{2}=e_{3},e_{3}^{2}=e_{4}^{2}=e_{5}.$

\item ${\mathcal{E}}%
_{5,15}:e_{1}^{2}=e_{2},e_{2}^{2}=e_{3},e_{3}^{2}=e_{4}^{2}=e_{5}.$

\item ${\mathcal{E}}%
_{5,16}:e_{1}^{2}=e_{2}^{2}=e_{3}^{2}=e_{4},e_{4}^{2}=e_{5}.$

\item ${\mathcal{E}}%
_{5,17}:e_{1}^{2}=e_{3}^{2}=e_{4},e_{2}^{2}=e_{4}+e_{5},e_{4}^{2}=e_{5}.$

\item ${\mathcal{E}}_{5,18}^{\alpha \in \mathbb{F}^{\ast }\mathbb{-}\left\{
1\right\} }:e_{1}^{2}=e_{4},e_{2}^{2}=e_{4}+e_{5},e_{3}^{2}=e_{4}+\alpha
e_{5},e_{4}^{2}=e_{5}.$

\item ${\mathcal{E}}%
_{5,19}:e_{1}^{2}=e_{2},e_{2}^{2}=e_{3}^{2}=e_{4},e_{4}^{2}=e_{5}.$

\item ${\mathcal{E}}%
_{5,20}:e_{1}^{2}=e_{2},e_{2}^{2}=e_{4},e_{3}^{2}=e_{4}+e_{5},e_{4}^{2}=e_{5}. 
$

\item ${\mathcal{E}}%
_{5,21}:e_{1}^{2}=e_{2}^{2}=e_{3},e_{3}^{2}=e_{4},e_{4}^{2}=e_{5}.$

\item ${\mathcal{E}}%
_{5,22}:e_{1}^{2}=e_{3},e_{2}^{2}=e_{3}+e_{5},e_{3}^{2}=e_{4},e_{4}^{2}=e_{5}. 
$

\item ${\mathcal{E}}_{5,23}^{\alpha \in \mathbb{F}%
}:e_{1}^{2}=e_{3}+e_{4},e_{2}^{2}=e_{3}+\alpha
e_{5},e_{3}^{2}=e_{4},e_{4}^{2}=e_{5}.$

\item ${\mathcal{E}}%
_{4,24}:e_{1}^{2}=e_{2},e_{2}^{2}=e_{3},e_{3}^{2}=e_{4},e_{4}^{2}=e_{5}.$

\item ${\mathcal{E}}%
_{4,25}:e_{1}^{2}=e_{3},e_{2}^{2}=e_{4},e_{3}^{2}=e_{4}^{2}=e_{5}.$

\item ${\mathcal{E}}_{4,26}:e_{1}^{2}=e_{4},e_{2}^{2}=e_{3}^{2}=e_{5}.$

\item ${\mathcal{E}}_{4,27}:e_{1}^{2}=e_{4},e_{2}^{2}=e_{3}^{2}=e_{4}+e_{5}.$

\item ${\mathcal{E}}_{4,28}:e_{1}^{2}=e_{2},e_{2}^{2}=e_{4},e_{3}^{2}=e_{5}.$

\item ${\mathcal{E}}%
_{4,29}:e_{1}^{2}=e_{3},e_{2}^{2}=e_{3}+e_{4},e_{3}^{2}=e_{5}.$
\end{itemize}

Among these algebras there is precisely the following isomorphism:

\begin{itemize}
\item ${\mathcal{E}}_{5,18}^{\alpha }\cong {\mathcal{E}}_{5,18}^{\beta }$ if
and only if $\left( \beta -\alpha \right) \left( \beta -\frac{1}{\alpha }%
\right) \left( \beta -1+\alpha \right) \left( \beta -\frac{1}{1-\alpha }%
\right) \left( \beta -\frac{\alpha -1}{\alpha }\right) \left( \beta -\frac{%
\alpha }{\alpha -1}\right) =0.$
\end{itemize}
\end{thm}


\begin{thebibliography}{9}
\bibitem{Elduque} A. Elduque, A. Labra, Evolution algebras and graphs,
Journal of Algebra and its Applications. 14 (2015) 1550103.

\bibitem{Hegazi} A. S. Hegazi, H. Abdelwahab, Nilpotent evolution algebras
over arbitrary fields, submitted to Linear Algebra App. (accepted)
arxiv:1508.06860.

\bibitem{V} J. Tian, P. Vojtechovsky, Mathematical concepts of evolution
algebras in non-Mendelian genetics, in: Quasigroups and Related Systems,
2006: pp. 111--122.

\bibitem{T} J. Tian, Evolution algebras and their applications, Springer,
Berlin, 2008.

\bibitem{Z} K. Zhevlakov, Rings that are nearly associative, Academic Press,
New York, 1982.
\end{thebibliography}
\end{document}